\documentclass[final]{ws-ccm}

\newcommand{\LL}{\mathop{\hbox{\vrule height 7pt width .5pt depth 0pt
\vrule height .5pt width 6pt depth 0pt}}\nolimits}

\newcommand{\req}[1]{(\ref{#1})}
\newcommand{\intt}[1]{\underset{#1}{\int}}

\newcommand{\R}{\mathbb R}

\newcommand{\NN}{\mathbb{N}}

\newcommand{\beq}{\begin{equation}}
\newcommand{\eeq}{\end{equation}}

\newcommand{\ds}{\displaystyle}

\begin{document}

\markboth{V. Millot \&  A. Pisante} {Symmetry of local minimizers for the three dimensional Ginzburg-Landau functional}

\title{SYMMETRY OF LOCAL MINIMIZERS 
FOR THE THREE DIMENSIONAL GINZBURG-LANDAU FUNCTIONAL}

\author{VINCENT MILLOT}

\address{D\'epartement de Math\'ematiques\\
Universit\'e de Cergy-Pontoise\\
2 avenue Adolphe Chauvin\\
95302 Cergy-Pontoise cedex, France\\
\emph{\tt{vmillot@math.u-cergy.fr}}}

\author{ADRIANO PISANTE}

\address{Department of Mathematics\\
University of Rome 'La Sapienza'\\
P.le Aldo Moro 5, 00185 Roma, Italy\\
\emph{\tt{pisante@mat.uniroma1.it}}}

\maketitle

\begin{abstract}
{\bf Abstract.} We classify nonconstant entire local minimizers of the standard Ginzburg-Landau functional  for maps
in ${H}^{1}_{\rm{loc}}(\mathbb{R}^3;\mathbb{R}^3)$ satisfying a natural energy bound.  
Up to  translations and rotations, such solutions of the Ginzburg-Landau system are given by an explicit solution equivariant under the action of the orthogonal group. 
\end{abstract}




\section{Introduction}

Symmetry results for nonlinear elliptic PDE's are difficult and usually rely on a clever use of the maximum principle  
as in the celebrated Serrin's moving planes method, or the use of rearrengement techniques as the Schwartz symmetrization (see, {\it e.g.}, \cite{B2} for a survey). 
In case of systems the situation is more involved since there are no general tools for proving this kind of results. 

In this paper we investigate symmetry properties of maps $u:\mathbb{R}^3 \to \mathbb{R}^3$ which are entire (smooth) solutions of the system
\begin{equation}
\label{GL}
\Delta u + u (1-|u|^2)=0
\end{equation}
possibly subject to the condition at infinity 
\begin{equation}
\label{modtoone}
|u(x)| \to 1 \quad \hbox{as} \quad |x|\to +\infty \, .
\end{equation}
The system \req{GL} is naturally  associated to the energy functional
\begin{equation}
\label{GLenergy}
E(v, \Omega):= \int_{\Omega} \bigg( \frac12 |\nabla v|^2 +\frac14 (1-|v|^2)^2 \bigg)dx 
\end{equation}
defined for $v\in H^1_{\rm loc}(\R^3;\R^3)$ and a bounded open set $\Omega\subset \R^3$. 
Indeed, if $u\in H^1_{\rm loc}(\R^3;\R^3)$ is a critical point of $E(\cdot,\Omega)$ for every $\Omega$ then $u$ 
is a weak solution of \req{GL} and thus a classical solution according to the standard regularity theory for elliptic equations. 
In addition, any weak solution $u$ of \req{GL} satisfies the natural bound $|u| \leq 1$ in the entire space, see \cite[Proposition 1.9]{F}.

Here the ``boundary condition" \req{modtoone} is added
to rule out solutions with values in a lower dimensional Euclidean space 
like the scalar valued solutions relevant for the De Giorgi conjecture (see, {\it e.g.}, \cite{AC}), or the explicit vortex solutions 
of \cite{HH} (see also \cite{BBH}) arising in the 2D Ginzburg-Landau model. 
More precisely, under the assumption \req{modtoone} the map $u$  has a well defined topological degree 
at infinity given by
$${\rm deg}_\infty u := {\rm deg} \bigg(\frac{u}{|u|}, \partial B_R \bigg)$$ 
whenever $R$ is large enough, and we are interested in solutions satisfying ${\rm deg}_\infty u \neq 0$. 
A special symmetric solution $U$ to \req{GL}-\req{modtoone} with ${\rm deg}_\infty U=1$ has been constructed in \cite{AF} and \cite{G} in the form
\begin{equation}
\label{GLsolutions}
U(x)=\frac{x}{|x|} f(|x|) \, ,
\end{equation}
for a unique function $f$ vanishing at zero and increasing to one at infinity. 
Taking into account the obvious invariance properties of \req{GL} and \req{GLenergy}, infinitely many solutions can be obtained from \req{GLsolutions} by translations on the domain 
and orthogonal transformations on the image. In addition, these solutions  satisfy $R^{-1}E(u,B_R) \to 4\pi$ as $R \to +\infty$. 
It is easy to check that $U$ as in \req{GLsolutions} is the unique solution $u$ of \req{GL}-\req{modtoone} such that $u^{-1}(\{0\})=\{ 0\}$, ${\rm deg}_\infty  u=1$ and 
$u$ is $O(3)$-equivariant, {\it i.e.}, $\displaystyle{u(Tx)=Tu(x)}$ for all $x \in \mathbb{R}^3$ and for all $T \in O(3)$ (see Remark \ref{equivariance}). 
In addition $u=U$ satisfies $|u(x)|=1+\mathcal{O}(|x|^{-2})$ as $|x|\to+\infty$.
\vskip5pt

In \cite{B2}, H. Brezis has formulated the following problem: 
\begin{itemize}
\item[] {\it Is any solution to \req{GL} satisfying \req{modtoone} (possibly with a ``good" rate of convergence) and ${\rm deg}_\infty u= \pm 1 $, of the form \req{GLsolutions} 
(up to a translation on the domain and an orthogonal transformation on the image)?}
\end{itemize}
 
\noindent In this paper we investigate this problem focusing on local minimizers of the energy in the following sense. 
 
\begin{definition}  
Let $u\in H^1_{\rm loc}(\R^3;\R^3)$. We say that $u$ is a local minimizer of $E(\cdot)$ if 
\begin{equation}
\label{minimality}
E(u,\Omega) \leq E(v,\Omega) 
\end{equation}
for any bounded open set $\Omega\subset\R^3$ and $v\in H^1_{\rm loc}(\R^3;\R^3)$ satisfying $v-u\in H^1_0(\Omega; \mathbb{R}^3)$.
\end{definition}
Obviously local minimizers are smooth entire solutions of \eqref{GL} but it is not clear that 
nonconstant local minimizers do exist or if the solutions obtained from \req{GLsolutions} are locally minimizing. 
In case of maps from the plane into itself the analogous problems 
are of importance in the study of the asymptotic behaviour of minimizers of the 2D Ginzburg-Landau energy near their vortices, 
the explicit solutions of the form \req{GLsolutions} giving the asymptotic profile of the minimizers in the vortex cores. Both these questions were essentially solved affirmatively in \cite{M1,M2,Sa} 
(see also \cite{R} for the more difficult gauge-dependent problem,
{\it i.e.}, in presence of a magnetic field) but the complete classification of entire solutions to \req{GL}-\req{modtoone}, even in the 2D case remains open. 
\vskip5pt

The first result of this paper concerns the existence of nonconstant local minimizers. 

\begin{theorem}
\label{existence}
There exists a smooth nonconstant solution $u:\R^3\to\R^3$ of  (\ref{GL})-(\ref{modtoone}) which is a local minimizer of $E(\cdot)$.  
In addition, $u(0)=0$,  ${\rm deg}_\infty u=1$ and  $R^{-1}E(u,B_R) \to 4\pi$ as $R \to +\infty$.
\end{theorem}

The construction of a nonconstant local minimizer relies on a careful analysis of the vorticity set for solutions $u_\lambda$ to
\begin{equation}
\label{GLball}
(P_\lambda) \quad 
\begin{cases}
\ds \Delta u+\lambda^2u(1-|u|^2)=0  & \quad \hbox{in} \, \, B_1\,,  \\
u={\rm Id}  & \quad \hbox{on} \, \, \partial B_1 \,,
\end{cases}\quad\lambda >0\,,
\end{equation}
which are absolute minimizers of the Ginzburg-Landau functional $E_\lambda(u,B_1)$ on $H^1_{\rm{Id}}(B_1;\mathbb{R}^3)$ where 
$$E_\lambda(u,\Omega):=\int_{\Omega}e_\lambda(u)dx \quad\text{with}\quad e_{\lambda}(u):=\frac12 |\nabla u|^2+\frac{\lambda^2}{4}(1-|u|^2)^2\,.$$ 
Up to a translation, we will obtain a locally  minimizing solution to \eqref{GL} 
as  a limit of $u_{\lambda_n}(x/\lambda_n)$ for some sequence $\lambda_n \to +\infty$. 
\vskip5pt

As the smooth entire solutions of \req{GL},  critical points of the energy functional $E_{\lambda}(\cdot,\Omega)$  
satisfy a fundamental monotonicity identity (see \cite{Sc}, \cite{LW2}). 

\begin{lemma}[Monotonicity Formula]\label{monform}
Assume that $u:\Omega\to\R^3$ solves $\Delta u+\lambda^2 u(1-|u|^2)=0$ in some  open set $\Omega\subset \R^3$ and $\lambda>0$. Then,   
\begin{multline}
\label{monotonicity}
\frac{1}{R}\,E_\lambda(u,B_{R}(x_0)) =\frac{1}{r}\,E_\lambda(u,B_{r}(x_0)) +\\
+\int_{B_{R}(x_0)\setminus B_{r}(x_0)}\frac{1}{|x-x_0|} \bigg| \frac{\partial u}{\partial |x-x_0|}\bigg|^2 dx
+ \frac{\lambda^2}{2}\int_{r}^{R}\frac{1}{t^2} \int_{B_t(x_0)}(1-|u|^2)^2dx\,dt \, ,
\end{multline}
for any $x_0 \in \Omega$ and any $0<r \leq R\leq {\rm dist}(x_0,\partial \Omega)$. 
\end{lemma}

An entire solution $u$ to \eqref{GL}  for which the left hand side of \req{monotonicity} (with $\lambda=1$) is bounded, {\it i.e.},  
\begin{equation}\label{lingro}
\sup_{R>0}R^{-1}E(u,B_R)<+\infty\,,  
\end{equation}
can be studied near infinity through a ``blow-down" analysis. 
More precisely, for each $R>0$ we introduce the scaled map $u_R$ defined by  
\begin{equation}\label{defscmap}
u_R(x):=u(Rx)\,,
\end{equation}
which is a smooth entire solution of 
\begin{equation}\label{GLresc}
\Delta u_R+R^2 u_R(1-|u_R|^2)=0 \, .
\end{equation}
Whenever $E(u,B_R)$ grows at most linearly with $R$, $E_R(u_R,\Omega)$ is equibounded and thus  $\{u_R\}_{R>0}$ is bounded in $H^1_{\rm{loc}}(\R^3;\R^3)$. 
Any weak limit $u_\infty:\R^3\to\R^3$ of $\{u_R\}_{R>0}$ as $R\to+\infty$ is called a tangent map to $u$ at infinity, and the potential term in the energy  forces $u_\infty$
to take values into $\mathbb{S}^2$. Moreover (see \cite{LW2}), $u_\infty$ turns out to be harmonic and positively $0$-homogeneous, {\it i.e.}, $u_\infty(x)=\omega(x/|x|)$ for some harmonic map 
$\omega:\mathbb{S}^2\to \mathbb{S}^2$, and 
$u_\infty$ is a solution or a critical point (among $\mathbb{S}^2$-valued maps) of
\[ \Delta v+v|\nabla v|^2=0 \, , \qquad E_\infty(v,\Omega)= \int_{\Omega} \frac12 |\nabla v|^2 dx \ , \]
respectively. This is readily the case for the equivariant solution \req{GLsolutions}, where $U_R(x) \to x/|x|$ strongly in $H^1_{\rm loc}(\R^3;\R^3)$ as $R \to+ \infty$. 
In the general case, uniqueness of the tangent map at infinity is not guaranteed and the possible lack of compactness of $\{u_R\}_{R>0}$ 
have been carefully analyzed in \cite{LW1,LW2} where the blow-up analysis 
of the defect measure arising in the limit of the measures $e_R(u_R)dx$ is performed. As a byproduct (see \cite[Corollary D]{LW2}), a quantization result for the normalized energy is obtained, 
namely $R^{-1}E(u,B_R) \to 4\pi k$ as $R \to +\infty$ for some $k \in\NN$, the case $k=1$ being valid both for the solution \req{GLsolutions} (see Proposition \ref{Radsol}) 
and the local minimizer constructed in Theorem \ref{existence}. The following result shows that the 
same property is true for any local minimizer of $E(\cdot)$ satisfying  \req{lingro}, so that any nonconstant local minimizer of $E(\cdot)$ satisfying \req{lingro}
realizes the lowest energy quantization level. 

\begin{theorem}
\label{quantization}
Let $u\in H^1_{\rm loc}(\R^3;\R^3)$ be a nonconstant local minimizer of $E(\cdot)$ satisfying \req{lingro}. 
Then  $R^{-1}E(u,B_R)\to 4\pi$ as $R \to+ \infty$ and the scaled maps $\{ u_R \}_{R>0}$ are relatively 
compact in $H^1_{\rm{loc}}(\R^3;\R^3)$. 
\end{theorem}

In proving this theorem, the first step is to apply the  blow-down analysis from infinity given in \cite{LW2}. 
Then, taking minimality into account, we exclude concentration by a comparison argument involving a ``dipole removing technique". 
This yields the compactness of the scaled maps. Finally another comparison argument based on minimality and on the results in \cite{BCL}, gives 
the desired value for the limit of the scaled energy. Here we believe that (as shown in \cite{Sa} for the 2D case) assumption \req{lingro} should always hold, as a consequence of local minimality.
\vskip5pt

In order to prove full symmetry of a nonconstant local minimizer, a natural approach is to prove uniqueness and symmetry of the tangent map at infinity, 
and then try to propagate the symmetry from infinity to the entire space. 
As a first step in this direction, we have the following result  
inspired by the asymptotic analysis developed for harmonic maps at isolated singularities 
in the important work \cite{Si1} (see also \cite{Si2},  \cite{Je} for a possibly simplified treatment and a more 
comprehensive exposition on the subject, and \cite{GW} for the case of $\mathbb{S}^2$-valued harmonic maps in $\mathbb{R}^3$). 

\begin{theorem}
\label{asymmetry}
Let $u$ be an entire smooth solution of \req{GL} satisfying  \req{lingro} and such that the scaled maps $\{u_R\}_{R>0}$  are relatively compact in $H^1_{\rm{loc}}(\R^3;\R^3)$.  
Then there exist a constant $C>0$ such that for all $x\in \mathbb{R}^3$, 
\begin{equation}
\label{simonbound}
|x|^2 (1-|u(x)|^2)+|x||\nabla u(x)|+|x|^3 |\nabla (1-|u(x)|^2)|+|x|^2 |\nabla^2 u(x)| \leq C \, ,
\end{equation}
and there exists a unique harmonic map $\omega:\mathbb{S}^2 \to \mathbb{S}^2$  such that ${\rm deg} \, \omega={\rm deg}_\infty u$ and 
setting $u_\infty(x)=\omega(x/|x|)$, 
 \begin{itemize}
 \item[(i)] $\| {u_R}_{|\mathbb{S}^2}-\omega \|_{C^2(\mathbb{S}^2;\mathbb{R}^3)}\to 0$ as $R\to +\infty\,$, \\
 \item[(ii)] $e_R(u_R)(x)dx \overset{*}{\rightharpoonup} \frac12 |\nabla u_\infty|^2dx $ weakly* as  measures as $R \to +\infty\,$.
 \end{itemize}
  If in addition ${\rm deg}_{\infty} u=\pm 1$ then $\omega(x)=Tx$ for some $T\in O(3)$.
\end{theorem}

This result strongly relies on the a priori 
bound \req{simonbound} for entire solutions to \req{GL} which, loosely speaking, do not exibit any bubbling phenomena at infinity (more precisely 
the scaled maps $\{u_R\}$ do not exibit energy concentration as $R\to+ \infty$).  
Whenever \eqref{simonbound} holds, we can write for $|x|$ sufficiently large  
the polar decomposition of the solution $u$ as $u(x)=\rho(x) w(x)$ for some positive function $\rho$ and some $\mathbb{S}^2$-valued map $w$ which have to solve the system
\begin{equation}
\label{polarsystem}
\begin{cases}
{\rm div} ( \rho^2(x) \nabla w(x) )+w(x) \rho^2(x) |\nabla w(x)|^2=0\,, \\
\Delta \rho (x)+\rho(x) (1-\rho^2(x))=\rho(x) |\nabla w(x)|^2\,,
\end{cases}
\end{equation}
for $|x|$ large. It is clear from \eqref{simonbound} that $\rho$ smoothly tends  to $1$ at infinity. Hence the unit map $w$ tends to be harmonic as $|x|\to +\infty$, and  
system \req{polarsystem} can be considered as a perturbation of the harmonic map system. In the present situation, uniqueness of the asymptotic limit can be obtained 
from an elementary but tricky  estimate on the radial derivative of $ w$, and we avoid the use of the Simon-Lojasievicz inequality.
\vskip5pt

Once the asymptotic symmetry is obtained we can adapt the division method used in \cite{M2} and \cite{R} to get full symmetry. The main result of the paper is the following.

\begin{theorem}
\label{SYMMETRY}
Let $u$ be an entire solution of \req{GL}. The following conditions are equivalent:
\begin{itemize}
\item[${(i)}$] $u$ is a nonconstant local minimizer of  $E(\cdot)$ satisfying \req{lingro};
\vskip5pt

\item[$(ii)$] $E(u,B_R)=4\pi R +o(R)$ as $R \to+\infty$;
\vskip5pt

\item[$(iii)$] $u$ satisfies $|u(x)|=1+\mathcal{O}(|x|^{-2})$ as $|x|\to +\infty$ and  ${\rm deg}_\infty  u=\pm 1$;
\vskip5pt 

\item[$(iv)$] up to a translation on the domain and an orthogonal transformation on the image, $u$ is $O(3)$-equivariant, i.e.,  $u=U$ as given by \req{GLsolutions}.  
\end{itemize}

\end{theorem}

As a consequence of this theorem, we see that under the assumption \req{lingro}, up to translations and orthogonal transformations,  any nonconstant local minimizer of 
$E_\lambda(\cdot)$ in $H^1_{\rm{loc}}(\mathbb{R}^3;\mathbb{R}^3)$ is given by $u(x)=U(\lambda x)$ with $U$ as in \req{GLsolutions}. 
In the limiting case $\lambda=+\infty$, a similar result has been proved in \cite[Theorem 2.2]{AL} showing that 
any nonconstant local minimizer $u$ of the Dirichlet integral $E_\infty(\cdot)$ in $H^1_{\rm{loc}}(\mathbb{R}^3;\mathbb{S}^2)$ is given by $u(x)=x/|x|$ 
up to translations and orthogonal transformations.
\vskip10pt

The plan of the paper is the following. In Section 2 we review the properties of the equivariant solution \req{GLsolutions}. 
In Section 3 we study minimizing solutions to $(P_\lambda)$ and prove Theorem \ref{existence}. 
In Section 4 we prove the quantization property for an arbitrary local minimizer, {\it i.e.}, we prove Theorem \ref{quantization}. 
In Section 5 we deal with asymptotic symmetry and Theorem \ref{asymmetry}. Finally  we obtain in Section 6 the full symmetry and the main result of the paper.


\section{The equivariant solution}\label{radsol}

In this section we collect some preliminary results  about equivariant entire solutions.
The existence statement and the qualitative study are essentially contained in \cite{AF,FG} and \cite{G}. In the following lemma we stress the asymptotic decay at infinity. 
\begin{lemma}
\label{radode}
There is a unique  solution $f \in C^2([0,+\infty))$ of 
\begin{equation}\label{cauchypb}
\begin{cases}
\ds f^{\prime \prime}+\frac{2}{r} f^\prime -\frac{2}{r^2} f +f(1-f^2)=0 \,,\\[8pt]
f(0)=0 \;\;\text{and}\;\; f(+\infty)=1\,.
\end{cases}
\end{equation}
In addition $0<f(r)<1$ for each $r>0$, $f^\prime(0)>0$, $f$ is strictly increasing,
\begin{equation}
\label{odedec}
R^2 |f^{\prime\prime}(R)|+R f^\prime(R)+\left| 2-R^2(1-f(R)^2)\right|=o(1) \quad \hbox{as} \quad R \to+\infty \, ,
\end{equation}
and 
\begin{equation}
\label{radendec}
\frac{1}{R}\int_0^R\bigg( \frac{r^2}{2}(f^\prime)^2+f^2 +r^2\frac{(1-f^2)^2}{4}\bigg) dr \to 1 \quad \hbox{as} \quad R \to +\infty \, .
\end{equation}
\end{lemma}

\begin{proof}
The existence of an increasing solution follows from \cite{G} and \cite{AF}. To obtain the estimates at infinity in \eqref{odedec}, we multiply the equation by $r^2 f^\prime (r)$ and an integration by parts yields 
\begin{equation}\label{multeq}
\frac{R^2}{2}(f^\prime(R))^2+\int_0^R r(f^\prime(r))^2 dr+\int_0^R r^2(1-(f(r))^2)f(r)f^\prime(r)dr=(f(R))^2\leq 1 \, .
\end{equation}
Using the monotony of $f$, we deduce that $ \int_0^{+\infty} r(f^\prime(r))^2 dr<+\infty$. Hence we can find a sequence $R_n\to+\infty$ such that $R_n f^\prime(R_n)\to 0$ as $n \to +\infty$. 
On the other hand the integral terms in \eqref{multeq}  admit a limit as $R \to +\infty$. As a consequence, $rf^\prime(r)$ admits a limit at infinity 
and thus $Rf^\prime(R)\to 0$ as $R\to +\infty$. For any  $k \in (0,1)$ fixed, multiplying the equation by $r^2$ and averaging over $(kR,R)$ leads to 
$$\frac{R^2f^\prime(R)-k^2R^2f^\prime(kR)}{(1-k)R}+\frac{1}{(1-k)R}\int_{kR}^R f(r)r^2(1-(f(r))^2)dr=\frac{2}{(1-k)R}\int_{kR}^R f(r)dr \,. $$
Since $f$ is increasing and tends to $1$ at infinity, we infer   
$$ k^2 \limsup_{R \to+ \infty} R^2 (1-(f(R))^2) \leq 2 \leq \liminf_{R\to +\infty} R^2 (1-(f(kR))^2) \, , $$
so that $R^2 (1-(f(R))^2) \to 2$ as $R \to +\infty$ by arbitrariness of $k$. Taking the equation into account  \req{odedec} follows. 
To prove \req{radendec} we multiply the equation by $r^2(1-f^2)$ and we integrate by parts on $(0,R)$ to get
$$R^2(1-(f(R))^2)f^\prime(R)+2\int_0^Rr^2f(f^\prime)^2dr+\int_0^R r^2f(1-f^2)^2dr=2\int_0^Rf(1-f^2)dr \, .$$
Since $f$ is increasing and tends to $1$ at infinity, we deduce using \req{odedec} that 
$$\frac{1}{R}\int_0^R r^2(1-f^2)^2dr+\frac{1}{R}\int_0^R 2r^2 (f^\prime)^2dr +R^2 (1-(f(R))^2)f^\prime(R) \to 0 \,,$$
and \req{radendec} follows easily.
\end{proof}

A consequence of the previous lemma is the following result.

\begin{proposition}
\label{Radsol}
Let $x_0 \in \mathbb{R}^3$ and $T \in O(3)$. Consider the function $f:[0,+\infty) \to [0,1)$ given by Lemma~\ref{radode}  and define 
$$w(x):= \frac{T(x-x_0)}{|x-x_0|} f(|x-x_0|)\,.$$ 
Then 
$w$ is a smooth solution of \eqref{GL}. In addition, $0<|w(x)|<1$ for each $x \neq x_0$, $w$ satisfies \req{simonbound} and  
\begin{equation}
\label{radscaledenergy}
 \lim_{R \to +\infty} \frac{1}{R} \int_{B_R(x_0)} \left( \frac{1}{2}|\nabla w(x)|^2+\frac{(1-|w(x)|^2)^2}{4}\right) dx = 4\pi  \, .
\end{equation} 
\end{proposition}
\begin{proof}
As in \cite{AF} and \cite{G}, $w$ is smooth and it is a classical solution of \eqref{GL}. It is routine to check that \req{simonbound} follows from \req{odedec}. 
Then a simple calculation yields 
$$\displaystyle{|\nabla w(x)|^2=(f^\prime(|x-x_0|))^2+\frac{2(f(|x-x_0|))^2}{|x-x_0|^2}+ \frac{(1-|f(|x-x_0|)|^2)^2}{4} }\,,$$ 
whence \req{radscaledenergy} follows from \req{radendec}. 
\end{proof}

\begin{remark}
\label{equivariance}
The solution $U$ given by \req{GLsolutions} is the unique $O(3)$-equivariant solution $u$ of \req{GL}-\req{modtoone} such that $u^{-1}(\{0\})=\{0\}$ 
and ${\rm deg}_\infty u=1$. 
Indeed for each fixed $x \neq 0$, setting $l_x$ to be the line passing through $0$ and $x$,  $u(l_x) \subset l_x$  because $u$ is equivariant (actually invariant) under rotations fixing $l_x$. 
Hence we can write $u(x)=(x/|x|)\sigma(x) |u(x)|$ with $\sigma(x)=\pm1$ and $|u(x)|=g(|x|)$ for some smooth function $g:(0,+\infty) \to (0,+\infty)$. 
Since $u$ is smooth and ${\rm deg}_\infty u=1$, we conclude that $\sigma\equiv 1$. 
Taking \req{modtoone} into account we conclude that $g$ satisfies the Cauchy problem \req{cauchypb}. Finally by the uniqueness result  in \cite{AF,G}, we obtain $g\equiv f$ as claimed.   
\end{remark}


\section{Existence of nonconstant local minimizers}

A basic ingredient in the construction of a nonconstant local minimizer is the following small energy regularity result taken from \cite{LW2} (see also \cite{CS}).

\begin{lemma}
\label{epsregularity}
There exists two positive constants $\eta_0>0$ and $C_0>0$ such that for any $\lambda\geq 1$ and any  $u\in C^2(B_{2R}(x_0);\mathbb{R}^3)$ satisfying 
\[ \Delta u +\lambda^2 u(1-|u|^2)=0 \quad \hbox{in $B_{2R}(x_0)$}\,, \]
with $\displaystyle{\frac{1}{2R}\,E_{\lambda}(u, B_{2R}(x_0)) \leq \eta_0 }\,$,   
then
\begin{equation}
\label{LinfL1energy}
R^2  \sup_{B_R(x_0)} e_{\lambda}(u) \leq C_0 \frac{1}{2R}\, E_{\lambda}(u,B_{2R}(x_0)) \, .
\end{equation}
\end{lemma}

We will also make use of the following boundary version of Lemma \ref{epsregularity} (see \cite{C,CL}). 
\begin{lemma}
\label{epsregularitybd}
Let $g:\partial B_1\to \mathbb{S}^2$ be a smooth map. There exists two positive constants $\eta_1>0$ and $C_1>0$ such that for any $\lambda\geq 1$, 
$0<R<\eta_1/2$, $x_0\in\partial B_1$ and any  $u\in C^2(\overline B_1\cap B_{2R}(x_0);\mathbb{R}^3)$  satisfying $u=g$ on $\partial B_1\cap B_{2R}(x_0)$ and 
\[ \Delta u +\lambda^2 u(1-|u|^2)=0 \quad \hbox{in $B_1\cap B_{2R}(x_0)$}\,, \]
with $\displaystyle{\frac{1}{2R}\,E_{\lambda}(u, B_1\cap B_{2R}(x_0)) \leq \eta_1 }\,$,   
then
\begin{equation}
\label{LinfL1energybis}
R^2  \sup_{B_1\cap B_R(x_0)} e_{\lambda}(u) \leq C_1\frac{1}{2R}\,E_{\lambda}(u, B_1\cap B_{2R}(x_0)) \, .
\end{equation}
\end{lemma}

Another result which is a combination  of \cite{LW1} and \cite{LW2}  will play a crucial role in the sequel.

\begin{proposition}
\label{linwang}
Let $\Omega \subset \mathbb{R}^3$ be a smooth bounded open set and let $\lambda_n\to +\infty$ as $n\to+\infty$. 
For every $n\in\NN$ let $u_n$  be a critical point of $E_{\lambda_n}(\cdot,\Omega)$ such that 
${\sup_{n} E_{\lambda_n}(u_n,\Omega) <+\infty}$. Then, up to a subsequence, $u_n \rightharpoonup u$ weakly in $H^1(\Omega;\mathbb{R}^3)$ for some 
weakly harmonic map $u:\Omega\to \mathbb{S}^2$ and  
$e_{\lambda_n}(u_n)(x)dx\overset{*}{\rightharpoonup} \frac12 |\nabla u|^2dx+\nu$ weakly* as measures on $\Omega$ where  
$\nu=4\pi \theta \mathcal{H}^1 \LL \Sigma$ for some $\mathcal{H}^1$-rectifiable set $\Sigma$ of locally finite $\mathcal{H}^1$-measure and some  integer valued 
$\mathcal{H}^1$-measurable function $\theta: \Sigma \to \mathbb{N}$.
\end{proposition}

The key result of this section is the following proposition. 

\begin{proposition}
\label{Vorticity}
Let $\lambda\geq 1$ and $u_\lambda \in H^1(B_1;\mathbb{R}^3)$ be a global minimizer of $E_{\lambda}(\cdot,B_1)$ over $H^1_{\rm{Id}}(B_1;\mathbb{R}^3)$. For any $\delta\in(0,1)$, 
there exists a constant $C_\delta>0$ independent of $\lambda$ such that ${\rm diam}\big(\{ |u_\lambda|\leq \delta \}\big )\leq C_\delta\lambda^{-1}$ and 
 ${\rm dist}_{H}\big(\{ |u_\lambda|\leq \delta \} ,\{0\}\big)=o(1)$ as $\lambda\to+\infty$ where ${\rm dist}_H$ denotes the Haussdorf distance. 
\end{proposition}

\begin{proof}  Let us consider an arbitrary sequence $\lambda_n\to +\infty$, and for every $n\in\NN$ let $u_n\in H^1(B_1;\R^3)$ be a global minimizer of 
$E_{\lambda_n}(\cdot,B_1)$ under the boundary condition ${u_n}_{|\partial B_1}=x$. It is well known that $u_n$ satisfies $u_n\in C^2(\overline B_1)$ 
and $|u_n|\leq 1$ for every $n\in\NN$. 
\vskip5pt 

\noindent {\em Step 1.} We claim that $u_n \to v(x):=x/|x|$ strongly in $H^1(B; \mathbb{R}^3)$. Since the map $v$ is admissible, one has 
\begin{equation}\label{bdenggl}
\frac{1}{2}\int_{B_1}|\nabla u_n|^2\leq E_{\lambda_n}(u_n,B_1)\leq E_{\lambda_n}(v,B_1)=\frac{1}{2}\int_{B_1}|\nabla v|^2=4\pi\quad\text{for every $n\in\NN$.}
\end{equation}
As a consequence, $\{u_n\}$ is bounded in $H^1(B_1;\R^3)$ and up to a subsequence, $u_n\to u_\star$ weakly in $H^1(B; \mathbb{R}^3)$ for some $\mathbb{S}^2$-valued map 
$u_\star$ satisfying ${u_\star}_{|\partial B_1}=x$. By Theorem~7.1 in \cite{BCL}, the map $v$ is the unique minimizer of $\,u\in H^1(B_1;\mathbb{S}^2)\mapsto \int_{B_1}|\nabla u|^2$ 
under the boundary condition $u_{|\partial B_1}=x$. In particular, 
$\int_{B_1}|\nabla u_\star|^2 \geq \int_{B_1}|\nabla v|^2$ 
which, combined with \eqref{bdenggl}, yields 
$$\frac{1}{2}\int_{B_1}|\nabla u_n|^2\to \frac{1}{2}\int_{B_1}|\nabla u_\star|^2=\frac{1}{2}\int_{B_1}|\nabla v|^2\quad\text{as $n\to+\infty$}\,.$$ 
Therefore $u_\star\equiv v$ and $u_n\to v$ strongly in $H^1(B; \mathbb{R}^3)$. 
\vskip5pt

\noindent {\em Step 2.} Let $\delta \in (0,1)$ be fixed. We now prove that the family of compact sets $\mathcal{V}_n:=\{|u_n| \leq \delta \} \to \{ 0 \}$ in the Hausdorff sense. 
It suffices to prove for any given $0<\rho<1$, $\mathcal{V}_n\subset B_\rho$ for every $n$ large enough. Since $v$ is smooth outside the origin, we can find $0<\sigma\leq\min( \rho/8,\eta_1/4)$ such that 
$$\frac{1}{\sigma}\int_{B_1\cap B_{4\sigma}(x)}|\nabla v|^2<\min(\eta_0,\eta_1):=\ell \quad \text{for every $x\in \overline B_1\setminus B_\rho$}\,,$$
where $\eta_0$ and $\eta_1$ are given by Lemma \ref{epsregularity} and Lemma  \ref{epsregularitybd} respectively. From the strong convergence of $u_n$ to $v$ in $H^1$, we infer that 
\begin{equation}\label{smalleng}
 \frac{1}{\sigma}\,E_{\lambda_n}(u_n,B_{4\sigma}(x))<\ell\quad \text{for every $x\in \overline B_1\setminus B_\rho$}
 \end{equation}
whenever $n \geq N_1$ for some integer $N_1$ independent of $x$.  
Next  consider a  finite family of points $\{x_j\}_{j\in J}\subset  \overline B_1\setminus B_\rho$ satisfying  $B_{2\sigma}(x_j)\subset B_1$ if $x_j\in B_1$ and 
$$\overline B_1\setminus B_\rho \subset \bigg(\bigcup_{x_j\in B_1}B_{\sigma}(x_j)\bigg)\cup \bigg(\bigcup_{x_j\in \partial B_1}B_{2\sigma}(x_j)\bigg)\,.$$ 
In view of \eqref{smalleng}, for each $j\in J$ we can apply Lemma \ref{epsregularity} in $B_{2\sigma}(x_j)$ if $x_j\in B_1$ and Lemma \ref{epsregularitybd} in $B_1\cap B_{4\sigma}(x_j)$ if $x_j\in \partial B_1$ to deduce  
$$\sup_{\overline B_1\setminus B_\rho}\,e_{\lambda_n}(u_n)\leq C\sigma^{-2}\quad\text{for every $n\geq N_1$}\,,  $$
for some constant $C$ independent of $n$. By Ascoli Theorem the sequence  $\{u_n\}$ is compact in $C^0(\overline{B_1} \setminus B_\rho)$, and thus $|u_n| \to 1$ uniformly in $\overline{B_1}\setminus B_\rho$. 
In particular $|u_n|>\delta$ in $\overline {B_1}\setminus B_\rho$ whenever $n$ is large enough. 
\vskip5pt

In the remaining of this proof we will establish the estimate $\text{diam}\,(\mathcal{V}_n)\leq C_\delta \lambda^{-1}_n$. We shall argue by contradiction. 
Setting $r_n:=\text{diam}\,(\mathcal{V}_n)$, we assume that for a subsequence $\kappa_n:=r_n\lambda_n\to +\infty$. Let $a_n,b_n\in \mathcal{V}_n$ 
such that $|a_n-b_n|=r_n$ and set $c_n$ to be the middle point of the segment $[a_n,b_n]$. In view of Step 2, we have $c_n\to 0$.  Next we define 
for $n$ large enough and $x\in B_{2}$, 
$$w_n(x):=u_n(r_nx+c_n)\,,$$
so that $w_n$ satisfies 
\begin{equation}\label{rescgl}
\Delta w_n+\kappa_n^2w_n(1-|w_n|^2)=0\quad \text{in $B_2$\,.}
\end{equation}
Up to a rotation, we may assume without loss of generality that $(a_n-c_n)/r_n=(1/2,0,0)=:P_1$ and $(b_n-c_n)/r_n=(-1/2,0,0)=:P_2$ so that 
\begin{equation}\label{dblevort}
|w_n(P_1)|=|w_n(P_2)|=\delta\quad \text{for every $n$ sufficiently large}\,. 
\end{equation}
\vskip5pt

\noindent {\em Step 3}.  We claim that up to a subsequence $w_n\to \phi$ strongly in $H^1_{\rm loc}(B_2;\mathbb{R}^3)$ for some weakly stationary 
harmonic map $\phi:B_2\to \mathbb{S}^2$. First we infer from \eqref{bdenggl} and the Mononocity Formula \eqref{monotonicity} applied to $w_n$ and $u_n$ that 
\begin{align}
\frac{1}{R}\,E_{\kappa_n}(w_n,B_R(x_0))\leq\frac{1}{1-|r_nx_0+c_n|}\,E_{\lambda_n}(u_n,B_{1-|r_nx_0+c_n|}(r_nx_0+c_n)) 
\label{estimon}\leq \frac{4\pi}{1-|r_n x_0+c_n|} \,,
\end{align}
for every $x_0\in B_2$ and $0<R<\text{dist}(x_0,\partial B_2)$. Hence $\sup_n E_{\kappa_n}(w_n,B_2)<+\infty$. 
In view of Proposition \ref{linwang}, up a further subsequence, $w_n\rightharpoonup \phi $ weakly in $H^1(B_2;\mathbb{R}^3)$  
where $\phi:B_2\to \mathbb{S}^2$ is a weakly harmonic map, and  
\begin{equation}\label{convrad}
e_{\kappa_n}(w_n)dx\mathop{\rightharpoonup}\limits^{*}\mu:= \frac{1}{2}|\nabla \phi|^2dx+\nu\quad\text{weakly* as measures on $B_2$}\,, 
\end{equation}
for some Radon measure $\nu=4\pi\theta \mathcal{H}^1\LL \Sigma$ where $\Sigma$ is a $\mathcal{H}^1$-rectifiable set with 
locally finite $\mathcal{H}^1$-measure and $\theta$ is an integer valued function. As a direct consequence of the Monotonicity Formula  \eqref{monotonicity} and \req{estimon}, we have 
\begin{equation}\label{estdefmeas}
\frac{1}{R}\, \nu(B_{R}(x_0))\leq \frac{1}{R} \,\mu(B_{R}(x_0))\leq 4\pi
\end{equation}
for every $x_0\in B_2$ and $0<R<\text{dist}(x_0,\partial B_2)$. By Theorem~2.83 in 
\cite{AFP}, the $1$-dimensional density of 
$\nu$ at $x_0$, {\it i.e.}, $\Theta_1(\nu,x_0)=\lim_{R\to0} (2R)^{-1}\nu(B_R(x_0))$, exists and coincides with $4\pi\theta(x_0)$ for $\mathcal{H}^1$-a.e. $x_0\in\Sigma$. In view of \eqref{estdefmeas} we 
deduce that $\theta\leq 1/2$ $\mathcal{H}^1$-a.e. on $\Sigma$. Since $\theta$ is integer valued, we have $\theta=0$ $\mathcal{H}^1$-a.e. on $\Sigma$, i.e.,  $\nu\equiv 0$. 
Going back to \eqref{convrad}, we conclude that $w_n\to \phi$ strongly in $H^1_{\rm loc}(B_2;\mathbb{R}^3)$ and 
\begin{equation}\label{convpen}
\kappa_n^2(1-|w_n|^2)^2\mathop{\longrightarrow}\limits_{n\to+\infty} 0\quad\text{in $L^1_{\rm loc}(B_2)$}\,. 
\end{equation} 
It now remains to prove the stationarity of $\phi$. Since $w_n$ is smooth and satisfies \eqref{rescgl}, we have 
$$\int_{B_2}e_{\kappa_n}(w_n)\,{\rm div}\,\zeta -\sum_{i,j=1}^3\frac{\partial \zeta_i}{\partial x_j}\,\frac{\partial w_n}{\partial x_i}\cdot \frac{\partial w_n}{\partial x_j}=0$$
 for every $\zeta\in C^1_c(B_2;\R^3)$. Using the local strong convergence of $w_n$ and \eqref{convpen}, we can pass to the limit $n\to+\infty$ in the above equation 
 to derive that 
 $$\int_{B_2}|\nabla \phi|^2\,{\rm div}\,\zeta -2\sum_{i,j=1}^3\frac{\partial \zeta_i}{\partial x_j}\,\frac{\partial \phi}{\partial x_i}\cdot \frac{\partial \phi}{\partial x_j}=0 \quad \forall \zeta\in C^1_c(B_2;\R^3)\,,$$
i.e., $\phi$ is stationary in $B_2$. 
\vskip5pt 

\noindent {\em Step 4.} By the energy monotonicity formula for stationary harmonic maps (see  \cite{Sc}) and \eqref{estimon}, we have
\begin{equation}\label{monhm}
\frac{1}{R_1}\int_{B_{R_1}(x_0)}|\nabla\phi|^2\leq \frac{1}{R_2}\int_{B_{R_2}(x_0)}|\nabla\phi|^2\leq 8\pi
\end{equation}
for every $x_0\in B_2$ and $0<R_1\leq R_2\leq {\rm dist}(x_0,\partial B_2)$. We claim that 
\begin{equation}\label{nonvanvor}
\lim_{R\to0}\frac{1}{R}\,\int_{B_R(P_i)}|\nabla \phi|^2=\inf_{0<R<1}\frac{1}{R}\,\int_{B_R(P_i)}|\nabla \phi|^2>0\quad \text{for $i=1,2$}\,.
\end{equation}
Indeed if the limit above vanishes, we could argue as in Step 2 using Lemma \ref{epsregularity} to deduce that $|w_n(P_i)|>\delta$ for $n$ large 
which contradicts \eqref{dblevort}. By the quantization results in \cite{LR}, for $i=1,2$, 
$$\lim_{R\to0}\frac{1}{R}\,\int_{B_R(P_i)}|\nabla \phi|^2 = 8\pi k_i \quad\text{for some $k_i\in\NN$}\,.$$
Combining  \eqref{monhm} with \eqref{nonvanvor}, we deduce that $k_1=k_2=1$ and thus
\begin{equation}\label{pricev}
\inf_{0<R<1}\frac{1}{R}\,\int_{B_R(P_i)}|\nabla \phi|^2=8\pi\quad \text{for $i=1,2$}\,.
\end{equation}
Setting $Q_R=(R-1/2,0,0)$ for $0<R<1$, we then have  
$$8\pi\geq \int_{B_1(Q_R)}|\nabla\phi|^2 \geq \int_{B_{R}(P_1)}|\nabla \phi|^2+\int_{B_{1-R}(P_2)}|\nabla \phi|^2\geq 8\pi R+8\pi(1-R)=8\pi\,. $$
Hence $|\nabla \phi|^2\equiv 0$ a.e. in $B_1(Q_R)\setminus \big(B_{R}(P_1)\cup B_{1-R}(P_2)\big)$ for every $0<R<1$. Since 
$$B_1\cap \bigcup_{0<R<1}\bigg( B_1(Q_R)\setminus \big(B_{R}(P_1)\cup B_{1-R}(P_2)\big)\bigg)=B_1\setminus [(-1,0,0),(1,0,0)]\,,$$
we derive that $\int_{B_1}|\nabla \phi|^2=0$ which obviously contradicts \eqref{pricev}. Therefore $r_n\lambda_n$ remains bounded and the proof is complete .
\end{proof}

\noindent {\bf Proof of Theorem \ref{existence}.} 
Consider a sequence $\lambda_n\to+\infty$ and let $u_n$ be a minimizer of $E_{\lambda_n}(\cdot,B_{1})$ on $H^1_{\rm{Id}}(B_{1};\mathbb{R}^3)$. 
By Proposition \ref{Vorticity},  $|u_n|\geq 1/2$ in $B_1\setminus B_{1/2}$ for $n$ large enough. In particular, $d_r:={\rm deg}(u_n, \partial B_r)$ is well defined for $1/2\leq r\leq1$  and  
$d_r=d_1=1$ thanks to the boundary condition. Hence we may find $a_n\in B_{1/2}$ such that $u_n (a_n)=0$ for every $n$ sufficiently 
large. Again by Proposition \ref{Vorticity}, $a_n\to 0$ and $\{|u_n|\leq 1/2\}\subset B_{r_n}(a_n)$ with $r_n:={\rm diam}(\{|u_n|\leq 1/2\})=\mathcal{O}(\lambda_n^{-1})$. 
Therefore ${\rm deg}(u_n, \partial B_r(a_n))=1$ for any $r\in [r_n, 1/2]$. 

Setting $R_n:=\lambda_n(1-|a_n|)$, $R_n \to +\infty$ as $n \to +\infty$, and we define for $x\in B_{R_n}$,  $\bar u_n(x):=u_n\big(\lambda_n^{-1}x+a_n\big)$ so that $\bar u_n$ satisfies  
$$ \Delta \bar u_n +\bar u_n(1-|\bar u_n|^2)\quad \text{in $B_{R_n}$}\,,$$
$\bar u_n(0)=0$ and $|\bar u_n|\leq 1$ for every $n$. Moreover arguing as in the previous proof, we obtain that 
\begin{equation}\label{quantconstr}
 \limsup_{n \to +\infty} R_n^{-1}E_1(\bar u_n,B_{R_n}) \leq 4\pi\,.
\end{equation}
Then we infer from standard elliptic theory that, up to a subsequence, $\bar u_n \to u$ in $C^2_{\rm loc}(\R^3)$ for some map $u:\R^3\to\R^3$ solving 
$\Delta u +u(1-|u|^2)=0$ in $\R^3$ and $u(0)=0$. By Proposition \ref{Vorticity} and the choice of $a_n$, we have $\{|\bar u_n|\leq 1/2\}\subset \overline B_{R_0}$ with 
$R_0:=\sup_n \lambda_n r_n<+\infty$. Hence  
$|u|\geq 1/2$ in  $\R^3\setminus B_{R_0}$ by continuity and locally uniform convergence. As a consequence, $u$ is nonconstant,  ${\rm deg}_{\infty} u$ is well defined and 
$${\rm deg}_{\infty} u={\rm deg}(u,\partial B_R)=\lim_{n\to+\infty} {\rm deg}(\bar u_n,\partial B_R)=\lim_{n \to+\infty} {\rm deg}(u_n, \partial  B_{r_n}(a_n))=1$$ 
for any $R\geq R_0$. Arguing in the same way, we infer from 
Proposition \ref{Vorticity} that $|u(x)|\to 1$ as $|x|\to +\infty$. Next we deduce from \eqref{quantconstr}, the Monotonicity Formula \req{monotonicity} 
and the smooth convergence of $\bar{u}_n$ to $u$, that $\sup_{R>0}\,R^{-1} E_{1}(u,B_R)\leq4\pi$. By the quantization result \cite[Corollary D]{LW2}, we have 
$R^{-1} E_{1}(u,B_R)\to 4\pi k$ as $R\to+\infty$ with $k\in\{0,1\}$. Since $u$ is nonconstant, we conclude that $k=1$. 
Finally, the local minimality of $u$ easily follows from the minimality of $u_n$ and the strong convergence in $H^1_{\rm loc}(\R^3;\mathbb{R}^3)$ of $\bar u_n$ to $u$. \prbox


\section{Energy quantization for local minimizers}

This section is devoted to the proof of Theorem \ref{quantization}.  
For any solution $u$ of \eqref{GL} satisfying \eqref{lingro}, 
the scaled maps  
$u_R(x):=u(Rx)$ are relatively weakly compact in $H^1_{\rm loc}(\R^3;\R^3)$. This fact will allow us to study such a map $u$ near infinity.     
First we recall that a tangent map to $u$ at infinity is a 
map $\phi:\R^3\to\R^3$ obtained as a weak limit of $u_n(x):=u(x/R_n)$ in $H^1_{\rm loc}(\R^3;\R^3)$ for some sequence of radii $R_n\to+\infty$. We denote by 
$\mathcal{T}_\infty(u)$ the set of all possible tangent  
maps to $u$ at infinity. The only information given by the potential at infinity is that any $\phi \in \mathcal{T}_\infty(u)$ takes values into $\mathbb{S}^2$. 
This is any easy consequence of the following elementary lemma which will be used in the sequel.

\begin{lemma}
\label{nopotential}
Let $u \in H^1_{\rm{loc}}(\mathbb{R}^3;\mathbb{R}^3)$ be a solution
of \req{GL} satisfying \req{lingro}. Then
\begin{equation}
\label{potentialto0}
\lim_{R \to +\infty} \frac{1}{R} \int_{B_R} \frac{(1-|u|^2)^2}{4}dx =0 \, .
\end{equation}
\end{lemma} 
\begin{proof}
We apply \req{monotonicity} with $\lambda=1$, $r>0$ and $R=2r$ to obtain
$$ \frac{1}{r} \int_{B_r} \frac{(1-|u|^2)^2}{4}dx \leq 4 \int_r^{2r} \frac{1}{t^2}\bigg( \int_{B_t} \frac{(1-|u|^2)^2}{4}\,dx\bigg) dt\leq \frac{1}{2r}E(u,B_{2r})-\frac{1}{r}E(u,B_r) \, . $$
Since the left hand side of \req{monotonicity} is bounded and increasing, the right hand side above tends to zero as $r$ tends to infinity and the conclusion follows.
\end{proof}

The following description of any tangent map has been obtained in \cite[Theorem~C]{LW2}.  

\begin{proposition}\label{descripblowdown}
Let $u$ be a solution of \eqref{GL} satisfying \req{lingro}. Let $\phi \in \mathcal{T}_\infty(u)$ and let $R_n\to+\infty$ be an associated sequence of radii. 
Then $\phi(x)=\phi(x/|x|)$ for $x\not=0$ and $\phi_{|\mathbb{S}^2}$ is a smooth harmonic map with values into $\mathbb{S}^2$. Moreover there exists a subequence (not relabelled) 
 such that 
\begin{equation}\label{defectcone}
e_{R_n}(u_n)dx \mathop{\rightharpoonup}\limits^{*} \frac{1}{2}|\nabla\phi|^2dx+\nu \quad\text{as $n\to+\infty$}\,,
\end{equation}
weakly* as measures for some nonnegative Radon measure $\nu$. In addition, if $\nu\not\equiv 0$ there exists an integer $1\leq l<\infty$, $\{P_j\}_{j=1}^l\subset \mathbb{S}^2$ 
and $\{k_j\}_{j=1}^l\subset \NN^*$ 
such that 
\begin{itemize}
\item[(i)] ${\rm spt}(\nu)=\cup_{j=1}^l\overline{OP_j}$ where $\overline{OP_j}$ denotes the ray emitting from the origin to $P_j$, and for $1\leq j\leq l$, 
$$\nu\LL\overline{OP_j}=4\pi k_j \mathcal{H}^1\LL\overline{OP_j}\,;  $$
\item[(ii)] the following balancing condition holds :
$$\frac{1}{2}\int_{\mathbb{S}^2}x|\nabla \phi|^2d\mathcal{H}^2 +4\pi\sum_{j=1}^l k_j P_j =0\,. $$
\end{itemize}
\end{proposition}

Under the assumption \req{lingro} we can apply Proposition \ref{descripblowdown} to any local minimizer of $E(\cdot)$. 
Now we claim that the local minimality of $u$ implies the strong convergence of the scaled maps $\{u_n\}$ to the associated tangent map.    

\begin{proposition}\label{propcomp}
Let $u\in H^1_{\rm loc}(\R^3;\R^3)$ be a local minimizer of $E(\cdot)$ satisfying \req{lingro}. Let $\phi \in \mathcal{T}_\infty(u)$ and let $R_n\to+\infty$ 
be the associated sequence of radii given by Proposition~\ref{descripblowdown}.       
Then $u_n\to \phi$ strongly in $H^1_{\rm loc}(\R^3)$ as $n\to+\infty$ and 
\begin{equation}\label{convmeas}
e_{R_n}(u_n)dx\mathop{\rightharpoonup}\limits^{*}\frac{1}{2}|\nabla\phi|^2dx
\end{equation}
weakly* as measures. 
\end{proposition}

\noindent {\bf Proof.}  In view of Proposition \ref{descripblowdown}, it suffices to prove that the defect measure $\nu$ in \eqref{defectcone} actually vanishes. We 
shall achieve it using a comparison argument. First we improve the convergence of $u_n$ away from ${\rm spt}(\nu)$. 
\vskip5pt

\noindent{\it Step 1.} First observe that $R_n^2(1-|u_n|^2)^2\to 0$ in $L^1_{\rm loc}(\R^3)$ by scaling and Lemma \ref{nopotential}. Next we claim 
that $u_n\to \phi$ in $C^1_{\rm loc}(\R^3\setminus({\rm spt}(\nu)\cup\{0\}))$.  Fix a ball 
$B_{4\delta}(x_0)\subset\subset \R^3\setminus({\rm spt}(\nu)\cup\{0\})$ with arbitrary center and $\delta$ to be chosen. Since $\phi$ is smooth away from the origin, 
we can choose $\delta$ small such that $\int_{B_{4\delta}(x_0)}|\nabla\phi|^2<4\delta \eta_0$ where the constant 
$\eta_0$ is given by Lemma \ref{epsregularity}. 
In view of \eqref{defectcone}, we have $\int_{B_{4\delta}(x_0)}e_{R_n}(u_n) \to \frac{1}{2}\int_{B_{4\delta}(x_0)}|\nabla\phi|^2$. In particular 
$\int_{B_{4\delta}(x_0)}e_{R_n}(u_n)\leq 4\delta\eta_0$ for $n$ large enough. By Lemma \ref{epsregularity}, we infer that $|\nabla u_n|\leq C_{\delta,x_0}$ and $|u_n|\geq 1/2$ 
in $B_{2\delta}(x_0)$ for $n$ large and a constant $C_{\delta,x_0}$ independent of $n$. Since $u_n$ satisfies \eqref{GLresc} (with $R=R_n$), 
setting $\rho_n:=1-|u_n|^2$, we have $0\leq \rho_n\leq 1$ and 
$-\Delta\rho_n+R_n^2\rho_n\leq 2 C^2_{\delta,x_0}$ in $B_{2\delta}(x_0)$. 
By a slight modification of Lemma 2 in \cite{BBH}, we infer that $\rho_n\leq C'_{\delta,x_0}R_n^{-2}$ in $B_\delta(x_0)$ for some constant $C'_{\delta,x_0}$ independent of $n$. 
Going back to \eqref{GLresc} we deduce that $|\Delta u_n|\leq C'_{\delta,x_0}$ in $B_\delta(x_0)$. 
Using standard $W^{2,p}_{\rm{loc}}$-regularity  
and the Sobolev embedding in $C^{1,\alpha}$-spaces, we finally conclude that $u_n\to \phi$ in $C^1(B_{\delta/2}(x_0))$. 
\vskip5pt

\noindent{\it Step 2.}  We will argue by contradiction and will assume that $\nu\not\equiv 0$ so that $k_1\geq 1$. 
Without loss of generality we may also assume that $P_1=(1,0,0)$ and $\phi(P_1)=(0,0,1)=:N$. We will construct for $n$ sufficiently large comparison maps $w_n$ which, roughly speaking, 
agree with $u_n$ except in a small cylinder around the $x_1$ axis, where they are constantly equal to $N$ and with smaller energy.
We consider two small parameters $0<\delta<<1$  and $0<\sigma<<1$.  In view of the explicit form of $\phi$ and $\nu$, we can find $x_\sigma \in \overline{OP_1}$ 
with $|x_\sigma|$ as large as needed   
such that $\overline Q_4(x_\sigma)\cap \overline{OP_j}=\emptyset$ for each $2\leq j\leq l$, 
\begin{equation}\label{controlphi}
\phi(Q_4(x_\sigma))\subset B_\sigma(N)\quad\text{and}\quad \int_{Q_4(x_\sigma)}|\nabla \phi|^2<\sigma\,.
\end{equation} 
Here we use the notation $Q_\rho(x_\sigma)=x_\sigma+\rho(-1/2,1/2)^3$ for $\rho>0$. Throughout the proof 
$T_\delta:=\R\times B^{(2)}_\delta(0)\subset\R^3$ will denote the infinite cylinder of size $\delta$ around the $x_1$ axis.  
In view of Step~1, for $n$ large enough 
\begin{equation}\label{controlun}
|u_n-\phi|<\sigma\quad\text{in $Q_4(x_\sigma)\setminus T_{\delta/2}$}\,, 
\end{equation}
and in particular  $|u_n|$ does not vanish in $Q_4(x_\sigma)\setminus T_{\delta/2}$ and it is actually as close to one as we want. 

Consider a cut-off function $\chi_1\in C^\infty_c(Q_4(x_\sigma);[0,1])$ satisfying $\chi_1\equiv 1$ in $Q_3(x_\sigma)$ 
and set  $\psi_\delta(x):=\min\{\delta^{-1}\chi_1(x)(2|x'|-\delta)^+, 1\}$ using the notation $x=(x_1,x')$. 
Then we define for $x\in Q_4(x_\sigma)$, 
$$\bar u_n(x) :=\psi_\delta (x)\,\frac{u_n(x)}{|u_n(x)|} +(1-\psi_\delta(x))u_n(x)\,.$$
Note that  $\bar u_n=u_n$ in a neighborhood of $\partial Q_4(x_\sigma)$, $\bar u_n=u_n$ in $Q_4(x_\sigma)\cap T_{\delta/2}$, and $(1-|\bar u_n|^2)^2\leq (1-|u_n|^2)^2$, 
because the double well potential is locally convex near its minima. 
Then we easily infer from Step 1 that $\bar u_n \to \phi$ in $W^{1,\infty}(Q_4(x_\sigma)\setminus T_{\delta/2})$ 
and
$$e_{R_n}(\bar u_n)dx\LL Q_4(x_\sigma)\mathop{\rightharpoonup}\limits^{*} \frac{1}{2}|\nabla\phi|^2dx\LL Q_4(x_\sigma)+\nu\LL Q_4(x_\sigma)$$
weakly* as measures. 
Now consider a  second cut-off function $\chi_2\in C^\infty_c(Q_3(x_\sigma);[0,1])$ satisfying $\chi_2\equiv 1$ in $Q_2(x_\sigma)$ 
and set $\tilde \psi_\delta(x)=\min\{\delta^{-1}\chi_2(x)(|x'|-\delta)^+, 1\}$. Define for $x\in Q_4(x_\sigma)$, 
$$v_n(x):=\begin{cases} 
\ds \frac{\tilde \psi_\delta(x) N+(1-\tilde \psi_\delta(x))\bar u_n(x)}{|\tilde \psi_\delta(x) N+(1-\tilde \psi_\delta(x))\bar u_n(x)|} & \text{if $x\in Q_3(x_\sigma) \setminus T_\delta$}\,,\\[8pt]
\bar u_n(x) & \text{if $x\in (Q_4(x_\sigma)\setminus Q_3(x_\sigma)) \cup (Q_4(x_\sigma)\cap T_\delta)$}\,,
\end{cases}$$
and 
$$\phi_\delta(x):= \frac{\tilde \psi_\delta(x) N+(1-\tilde \psi_\delta(x))\phi(x)}{|\tilde \psi_\delta(x) N+(1-\tilde \psi_\delta(x))\phi(x)|}\,.$$
Note that $\phi_\delta$ and $v_n$ are well defined  and smooth (Lipschitz) thanks to \eqref{controlphi} and \eqref{controlun}. Moreover 
$v_n=u_n$ both in a neighborhood of $\partial Q_4(x_\sigma)$ and in $Q_4(x_\sigma) \cap T_{\delta/2}$, and $v_n\equiv N$ in $Q_2(x_\sigma)\setminus T_{2\delta}$. 
From the construction of $\bar u_n$, we derive that 
$v_n \to\phi_\delta $ in $W^{1,\infty}(Q_4(x_\sigma)\setminus T_{\delta/2})$ and 
\begin{equation}\label{concvn}
e_{R_n}(v_n)dx\LL Q_4(x_\sigma)\mathop{\rightharpoonup}\limits^{*} \frac{1}{2}|\nabla\phi_\delta|^2dx\LL Q_4(x_\sigma)+\nu\LL Q_4(x_\sigma)
\end{equation}
weakly* as measures. Since $\nu$ does not charge the boundary of $Q_\rho(x_\sigma)$ for every $\rho>0$, we have 
$$\int^{|x_\sigma|+1}_{|x_\sigma|+1/2}\bigg(\int_{\{x_1=r\}\cap T_{2\delta}}e_{R_n}(v_n)\bigg)dr\mathop{\longrightarrow}\limits_{n\to+\infty} 
\frac{1}{2}\int_{\{|x_\sigma|+1/2<x_1<|x_\sigma|+1\}\cap T_{2\delta}}|\nabla\phi_\delta|^2 +2\pi k_1\,.$$ 
On the other hand, one may derive from the explicit form of $\phi_\delta$ and \eqref{controlphi} that 
\begin{equation}\label{smallphidel}
\int_{Q_4(x_\sigma)}|\nabla\phi_\delta|^2\leq C_\delta \sigma\,,
\end{equation}
where $C_\delta$ denotes a constant independent of $\sigma$. 
Hence we can find $r_n^+\in [|x_\sigma|+1/2,|x_\sigma|+1]$ such that 
$$\limsup_{n\to +\infty} \int_{\{x_1=r_n^+\}\cap T_{2\delta}}e_{R_n}(v_n)\leq 4\pi k_1 +C_\delta \sigma\,.$$
Arguing in the same way, we find $r_n^-\in[|x_\sigma|-1,|x_\sigma|-1/2]$ such that 
$$\limsup_{n\to +\infty} \int_{\{x_1=r_n^-\}\cap T_{2\delta}}e_{R_n}(v_n)\leq 4\pi k_1 +C_\delta \sigma\,.$$
Next we introduce the sets 
\begin{align*}
C_n^+&:= T_{2\delta}\cap\big\{r_n^+-2\delta\leq x_1\leq r_n^+\,,\,|x'|\leq x_1-(r_n^+-2\delta)\big\}\,,\\
C_n^-&:= T_{2\delta}\cap\big\{r_n^-\leq x_1\leq r_n^-+2\delta\,,\,|x'|\leq (r_n^-+2\delta)-x_1\big\}\,,\\
D_n&:=T_{2\delta}\cap \{x\in T_{2\delta},\,x_1\in(r_n^-,r_n^+) \}\,.
\end{align*}
Define for $x\in Q_4(x_\sigma)$ and $n$ large enough, 
$$w_n(x)= \begin{cases}
v_n(x) & \text{if  $x\in Q_4(x_\sigma)\setminus D_n$}\\[8pt]
\ds v_n\bigg(r_n^+,\frac{2\delta x'}{x_1-(r_n^+-2\delta)}\bigg) &\text{if $x\in C_n^+$\,,}\\[10pt]
\ds v_n\bigg(r_n^-,\frac{2\delta x'}{(r_n^-+2\delta)-x_1}\bigg) &\text{if $x\in C_n^-$\,,}\\[8pt]
N & \text{if $x\in D_n\setminus (C_n^+\cup C_n^-)$\,.}
\end{cases}
$$
One may check that $w_n\in H^1(Q_4(x_\sigma);\R^3)$ and $w_n=u_n$ in a neighborhood of $\partial Q_4(x_\sigma)$. Moreover,  straightforward 
computations yield
$$\int_{C_n^+}e_{R_n}(w_n)\leq C\delta \int_{\{x_1=r_n^+\}\cap T_{2\delta}} e_{R_n}(v_n)\quad\text{and}\quad 
\int_{C_n^-}e_{R_n}(w_n)\leq C\delta \int_{\{x_1=r_n^-\}\cap T_{2\delta}} e_{R_n}(v_n)\,,$$
for some absolute constant $C$. Recalling \eqref{concvn}, \eqref{smallphidel}, the fact that $\nu$ does not charge the boundary of $Q_\rho(x_\sigma)$ for every $\rho>0$ and 
 $Q_4(x_\sigma)=(Q_4(x_\sigma) \setminus D_n) \cup (C_n^+\cup C_n^-) \cup (D_n \setminus (C_n^+\cup C_n^-)) $, we finally obtain
\begin{equation}\label{enegcompmap}
\limsup_{n\to+\infty}\int_{Q_4(x_\sigma)}e_{R_n}(w_n) \leq 12\pi k_1+C\delta+C_\delta \sigma \,, 
\end{equation}
for some constant $C$ independent $\sigma$ and $\delta$, and some constant $C_\delta$ independent of $\sigma$. 
\vskip5pt

\noindent{\it Step 3.} From the local minimality of $u$, we infer that 
$$\int_{Q_4(x_\sigma)}e_{R_n}(u_n)\leq  \int_{Q_4(x_\sigma)}e_{R_n}(w_n) \,.$$
Using \eqref{defectcone} and \eqref{enegcompmap} we let $n\to+\infty$ in the above inequality to derive 
$$16 \pi k_1 \leq \nu(Q_4(x_\sigma))+\int_{Q_4(x_\sigma)}\frac12 |\nabla \phi|^2 dx=\lim_{n \to \infty} \intt{Q_4(x_\sigma)} e_{R_n}(u_n) \leq  12\pi k_1+C\delta+C_\delta \sigma\,.$$
Passing successively to the limits $\sigma\to0$ and $\delta\to 0$, we conclude that $k_1=0$. This contradicts our assumption $k_1\geq 1$ and the proof is complete.   
\prbox

\begin{corollary}\label{strtangmap}
Let $u\in H^1_{\rm loc}(\R^3;\R^3)$ be a nonconstant local minimizer of $E(\cdot)$ satisfying \req{lingro}. Then any $\phi \in \mathcal{T}_\infty(u)$ 
is of the form $\phi(x)=Tx/|x|$ for some $T\in O(3)$. 
\end{corollary}

\noindent {\bf Proof.} {\it Step 1.} 
First we claim that any $\phi \in \mathcal{T}_\infty(u)$ is energy minimizing in $B_1$, {\it  i.e.},
 \begin{equation}\label{minphi}
 \int_{B_1}|\nabla\phi|^2dx \leq \int_{B_1}|\nabla \varphi |^2dx \quad\text{for all $\varphi\in H^1(B_1;S^2)$ such that $\varphi_{|\partial B_1}=\phi$}\,.
 \end{equation}
Let $R_n\to+\infty$ be the sequence of radii  given by Proposition \ref{descripblowdown}, and let $\{u_n\}$ be the associated sequence of scaled maps. It follows from Step 2 in the 
previous proof that 
$$ \int_{B_1}e_{R_n}(u_n)dx \to \frac{1}{2}\int_{B_1}|\nabla \phi|^2dx$$
as $n\to+\infty$. In particular,
\begin{equation}\label{vanpot}
R_n^{2}\int_{B_1}(1-|u_n|^2)^2dx\to 0\,.
\end{equation}
In view of the local minimality of $u$, it suffices to prove that for 
any $\varphi \in H^1_\phi(B_1;S^2)$, there exists a sequence $\varphi_n\in H^1_{u_n}(B_1;\R^3)$ such that 
\begin{equation}\label{approx}
 \int_{B_1}e_{R_n}(\varphi_n)dx \to \frac{1}{2}\int_{B_1}|\nabla\varphi|^2dx\,.
 \end{equation}
We proceed as follows. From the previous proof we know that $u_n\to \phi$ uniformly  in the annulus $K:=\overline B_1\setminus B_{1/2}$. In particular, 
$|u_n|\geq 1/2$ in $K$ for $n$ large and setting $v_n:=u_n/|u_n|$, 
$$\delta_n:=\|v_n-\phi\|_{L^\infty(K)}+\|1-|u_n|^2\|_{L^\infty(K)}\mathop{\longrightarrow}\limits_{n\to+\infty} 0\,. $$
Denote $\mathcal{D}:=\{(s_0,s_1)\in\mathbb{S}^2\times \mathbb{S}^2\,,\,|s_0-s_1|<1/4\}$ and consider a continuously differentiable mapping 
$\Pi:\mathcal{D}\times [0,1]\to \mathbb{S}^2$ satisfying 
$$\Pi(s_0,s_1,0)=s_0\,,\quad \Pi(s_0,s_1,1)=s_1\,,\quad \bigg|\frac{\partial\Pi}{\partial t}(s_0,s_1,t)\bigg|\leq C|s_0-s_1|\,,$$
{\it e.g.},  the map giving geodesic convex combinations between points $s_0$ and $s_1$ on $\mathbb{S}^2$.

Given $\varphi \in H^1_\phi(B_1;S^2)$, we define for $n$ large enough, 
$$\varphi_n(x)=\begin{cases}
\ds\varphi\bigg(\frac{x}{1-2\delta_n}\bigg) & \text{for $x\in B_{1-2\delta_n}$}\,,\\[10pt]
\ds \Pi\bigg(v_n(x),\phi(x), \frac{1-\delta_n-|x|}{\delta_n}\bigg) & \text{for $x\in B_{1-\delta_n}\setminus B_{1-2\delta_n}$}\,,\\[10pt]
\ds \bigg(\frac{1-|x|}{\delta_n}+|u_n(x)|\frac{|x|-1+\delta_n}{\delta_n}\bigg)v_n(x) & \text{for $x\in B_1\setminus B_{1-\delta_n}$}\,.
\end{cases}
$$
One may easily check that $\varphi_n\in H^1(B_1;\R^3)$ and that 
\begin{equation}\label{energapprox}
\int_{B_1}e_{R_n}(\varphi_n)dx=\frac{1-2\delta_n}{2}\int_{B_1}|\nabla\varphi|^2dx+\frac{1}{2}\int_{B_{1-\delta_n}\setminus B_{1-2\delta_n}}|\nabla \varphi_n|^2dx+\int_{B_1\setminus B_{1-\delta_n}}e_{R_n}(\varphi_n)dx\,.
\end{equation}
Straighforward computations yield
$$\int_{B_{1-\delta_n}\setminus B_{1-2\delta_n}}|\nabla \varphi_n|^2dx \leq C \int_{B_{1-\delta_n}\setminus B_{1-2\delta_n}}\bigg(|\nabla \varphi|^2+|\nabla u_n|^2+\delta_n^{-2}|
v_n-\phi|^2\bigg)dx\mathop{\longrightarrow}\limits_{n\to+\infty} 0\,,$$
and 
$$\int_{B_{1}\setminus B_{1-\delta_n}}e_{R_n}(\varphi_n)dx \leq C \int_{B_{1}\setminus B_{1-\delta_n}}\bigg(|\nabla u_n|^2+(\delta_n^{-2}+R_n^{2})
(1-|u_n|^2)^2\bigg)dx\mathop{\longrightarrow}\limits_{n\to+\infty} 0\,, $$
where we used the fact $(1-|\varphi_n|^2)^2\leq (1-|u_n|^2)^2$ for $n$ large enough, again by convexity of the double well potential near its minima, and \eqref{vanpot} in the last estimate. 
In view of  \eqref{energapprox}, it completes the proof of \eqref{approx}. 
\vskip5pt

\noindent{\it Step 2.} In view of the monotonicity with respect to $R$ of  $R^{-1}E(u,B_R)$, if $u$ is nonconstant then \eqref{convmeas} yields 
\begin{equation}\label{limeng}
0<\lim_{R\to+\infty} R^{-1}E(u,B_{R})=\lim_{n\to+\infty} E_{R_n}(u_n,B_1)=\frac{1}{2} \int_{B_1}|\nabla\phi|^2dx\,,
\end{equation}
and thus $\phi$ is nonconstant.  Then the conclusion follows from Theorem 7.3 and Theorem 7.4 in \cite{BCL} together with \eqref{minphi}.  
\prbox
\vskip10pt

\noindent{\bf Proof of Theorem \ref{quantization}.} Let $R_n\to+\infty$ be an arbitrary sequence of radii. By  \req{lingro}, Proposition \ref{descripblowdown}, 
Proposition \ref{propcomp} and Corollary \ref{strtangmap}, we can find a subsequence (not relabelled) and $T\in O(3)$ such that the sequence of scaled maps $u_n(x)=u(R_nx)$ 
converges strongly in $H^1_{\rm loc}(\R^3;\R^3)$ to $\phi(x)=Tx/|x|$. Therefore  \eqref{limeng} gives 
$R^{-1}E(u,B_{R})\to 4\pi$ as $R\to+\infty$, 
and the proof is complete.\prbox


\section{Asymptotic symmetry}
In order to study the asymptotic behaviour of local minimizers we first derive some decay properties of solutions to \req{GL} at infinity. 
It will be clear that the crucial ingredients are \req{lingro}, the $H^1_{\rm{loc}}(\R^3;\R^3)$ compactness of the scaled maps and the small energy regularity lemma recalled in Section 3. 
Then we bootstrap the first order estimates to get higher order estimates and compactness of the rescaled maps and their derivatives of all orders. 
Finally we prove a decay property of the radial derivative 
which will give uniqueness of the asymptotic limit at infinity in the $L^2$-topology, whence uniqueness of the limit in any topology follows.
\vskip5pt

We start with the following result.

\begin{proposition}
\label{Firstderbounds}
Let $u$ be a smooth solution to \req{GL} satisfying \req{lingro} and such that the scaled maps $\{u_R \}_{R>0}$ are relatively 
compact in $H^{1}_{\rm{loc}}(\mathbb{R}^3;\mathbb{R}^3)$. 
Then there is a constant  $C>0$ such that for all $x\in\R^3$, 
\begin{equation}
\label{1derbound} 
|x|^2(1-|u(x)|^2)+|x||\nabla u(x)| \leq C \, .
\end{equation}
\end{proposition}

\begin{proof}
We prove the statement by contradiction. Assume \req{1derbound} were false, 
then there would be a sequence $\{ x_n \} \subset \mathbb{R}^3$ such that $R_n=|x_n| \to +\infty$ as $n \to +\infty$ and
\begin{equation}\label{hypcontr}
|x_n| |\nabla u (x_n)|+|x_n|^2 (1-|u(x_n)|^2) \mathop{\longrightarrow}\limits_{n \to +\infty} +\infty  \, .
\end{equation}
 For each integer $n$, let us consider $u_n(x):=u_{R_n}(x)=u(R_nx)$
as a entire solution of \eqref{GLresc}. Up to the extraction of a subsequence,  we may assume  that 
$x_n/R_n\to \bar{x} \in \partial B_1$ as $n \to +\infty$.  
By Proposition \ref{descripblowdown}, up to a further subsequence the sequence of scaled maps $\{u_n\}$ converges to $u_\infty(x)= \omega \left(x/|x|\right)$ strongly 
in $H^1_{\rm{loc}}(\mathbb{R}^3;\mathbb{R}^3)$ as $n\to +\infty$, where $\omega:\mathbb{S}^2 \to \mathbb{S}^2$ is an harmonic map. 
In addition 
$e_{R_n}(u_n)(x)dx \overset{*}{\rightharpoonup} \frac12 |\nabla u_\infty|^2dx+\nu$ where $\nu$ is a quantized cone-measure. 
Combining this property together with the strong convergence in $H^1_{\rm{loc}}(\R^3;\R^3)$ and Lemma \ref{nopotential}, we conclude that $\nu\equiv0$. 
Since $\omega$ is a smooth map we have $u_\infty \in C^\infty(\mathbb{R}^3 \setminus \{ 0\};\mathbb{S}^2)$. 
In particular $u_\infty$ is smooth around $\bar{x} \in\partial B_1$. Now we can argue as in Step 1 in the proof of Proposition \ref{propcomp} to find $\delta >0$ such that 
$|\nabla u_n|+R_n^2(1-|u_n|^2)\leq C_\delta$ in $B_\delta(\bar x)$ for some constant $C_\delta$ independent of $n$. Scaling back we obtain for $n$ large enough, 
$$ |x_n| |\nabla u (x_n)|+|x_n|^2 (1-|u(x_n)|^2) \leq C_\delta\,,$$ 
which obviously contradicts \eqref{hypcontr}.
\end{proof}

\begin{remark}
\label{modtoonerate}
For an arbitrary entire solution $u$ to \eqref{GL}, the estimate \req{1derbound} still holds under the assumption $|u(x)|=1+\mathcal{O}(|x|^{-2})$ as $|x|\to +\infty$. 
Indeed, since the scaled map $u_R$ given by \eqref{defscmap} satisfies \eqref{GLresc}, 
$\{\Delta u_R\}_{R>0}$ is equibounded in $L^\infty_{\rm{loc}} (\mathbb{R}^3 \setminus \{ 0 \})$. Therefore standard $W^{2,p}_{\rm{loc}}$ estimates 
and the Sobolev embedding show that $\{\nabla u_R\}_{R>0}$ is equibounded in $L^\infty_{\rm{loc}} (\mathbb{R}^3 \setminus \{ 0 \})$ 
which proves \req{1derbound}. Note also that \req{1derbound} implies \req{lingro}. 
\end{remark}

For a solution $u$ to \eqref{GL} satisfying  the assumptions of Proposition \ref{Firstderbounds}, 
we have $|u(x)|=1 +\mathcal{O}(|x|^{-2})$ and $|\nabla u (x)| =\mathcal{O}(|x|^{-1})$ as $|x| \to +\infty$. 
In order to get bounds on the higher order derivatives of $u$ at infinity it is very convenient to use the polar 
decomposition for $u$, {\it i.e.}, to write $u=\rho w$ for some nonnegative function $\rho$ and some $\mathbb{S}^2$-valued map $w$. 
The following result gives the $3D$ counterpart of the asymptotic estimates of \cite{S} 
for the $2D$ case, and it is essentially based on the techniques introduced in the proof of \cite{BBH2}, Theorem 1.

\begin{proposition}
\label{BBH}
Let $u$ be an entire solution of \req{GL} satisfying \req{1derbound}. Let $R_0 \geq 1$ be such that $|u(x)| \geq 1/2$ for $|x|\geq R_0/4$. For 
$R \geq R_0$ and $|x|\geq 1/4$, define $u_R(x)=u(Rx)= \rho_{R}(x) w_R(x)$ the polar decomposition of the scaled maps, i.e.,  
$\rho_R(x):=|u_R(x)|$ and $w_R(x):=u_R(x)/|u_R(x)|$.  Then for each $k \in\NN$ and each $\sigma \in (1,2)$ 
there exist constants $C=C(k,\sigma)>0$ and $C^\prime=C^\prime(k,\sigma)>0$ independent of $R$ such that
\begin{equation}
\label{polarderbounds}
\begin{array}{ll}
(\mathcal{P}_k^\prime) \qquad \qquad &  \| \nabla w_R \|_{C^k\big(\overline B_{2\sigma} \setminus B_{1/{2\sigma}}\big)} \leq C^\prime(k,\sigma) \,, \\[8pt]
 (\mathcal{P}_k^{\prime\prime}) \qquad \qquad &  \| R^2 (1- \rho_R) \|_{C^k\big(\overline B_{2\sigma} \setminus B_{1/{2\sigma}}\big)} \leq C^{\prime \prime}(k,\sigma) \, .
\end{array}
\end{equation}
As a consequence,  for each $k \in\NN$ there is a constant $C(k)>0$ such that
\begin{equation}
\label{highderbound}
\sup_{x \in \mathbb{R}^3} \big( |x|^{k+1} |\nabla^{k+1} u(x)|+|x|^{k+2} |\nabla^{k} (1-|u(x)|^2)|  \big) \leq C(k) \, .
\end{equation}
 \end{proposition}

\begin{proof}
Observe that it is suffices to prove \req{polarderbounds} since \req{highderbound} follows by scaling. 
For $|x| \geq R_0/4$ we have $|u(x)| \geq 1/2$ so we can write and $u(x)=\rho(x) w(x)$ with $\rho(x):=|u(x)|$ and $ w(x):=u(x) \rho(x)^{-1}$ and 
the system \req{polarsystem} is
satisfied in $\R^3\setminus \overline B_{R_0/4}$.   Hence, 
for each $R\geq R_0$ the scaled maps $u_R$, $\rho_R$ and $w_R$ are well defined and smooth in $\R^3\setminus \overline B_{1/4}$. In addition, 
\req{polarsystem} yields by scaling the following Euler Lagrange equations,
\begin{equation}
\label{scaledpolarsystem}
\begin{cases}
{\rm div} (\rho_R^2\nabla w_R )+w_R \rho_R^2|\nabla w_R|^2=0  \\[5pt]
\Delta \rho_R+\rho_R R^2(1-\rho_R^2)=\rho_R |\nabla w_R|^2 
\end{cases}
\quad \text{in $\R^3\setminus \overline B_{1/4}$}\,.
\end{equation}
We will prove \req{polarderbounds} by induction over $k$, the case $k=0$ being easily true by assumption \req{1derbound}. We closely follow 
\cite[pg. 136-137]{BBH2} with minor modifications.
\vskip5pt

First we prove that $(\mathcal{P}_k^\prime)$-$(\mathcal{P}_{k}^{\prime \prime})$ implies $(\mathcal{P}_{k+1}^\prime)$. We set for simplycity
\begin{equation}
\label{X_R}
X_R:=R^2(1- \rho_R) \, ,
\end{equation}
so that the second equation in \req{scaledpolarsystem} can be rewritten as 
\begin{equation}
\label{eqrhoR}
-\Delta \rho_R=-\rho_R  |\nabla w_R|^2+ \rho_R (1+\rho_R) X_R \, .
\end{equation}
By the inductive assumptions \req{polarderbounds} the right hand side in \req{eqrhoR} is bounded in $C^k_{\rm{loc}}(B_4 \setminus \overline B_{1/4})$ uniformly 
with respect to $R \geq R_0$. Hence $\{\rho_R\}_{R\geq R_0}$ is bounded in $W^{k+2,p}_{\rm{loc}}(B_4 \setminus \overline B_{1/4})$ for each $p<+\infty$ 
by standard elliptic regularity theory. Then the Sobolev embedding implies that 
$\{\nabla \rho_R\}_{R\geq R_0}$ is also bounded in $C^k_{\rm{loc}}(B_4 \setminus \overline B_{1/4})$. 
Next rewrite the first equation in \req{scaledpolarsystem} as
\begin{equation}
\label{eqwR}
-\Delta w_R=w_R |\nabla w_R|^2+\frac{2\nabla \rho_R}{\rho_R} \nabla w_R \, .
\end{equation}
Since all the terms in the right hand side in \eqref{eqwR} are now 
bounded in $C^k_{\rm{loc}}(B_4 \setminus \overline B_{1/4})$
uniformly with respect to $R\geq R_0$, standard linear theory (differentiating the equation $k$-times)  also gives that 
$\{w_R\}_{R\geq R_0}$ is equibounded  in $W^{k+2,p}_{\rm{loc}}(B_4 \setminus \overline B_{1/4})$ for each $p<+\infty$. Therefore the right hand side in 
\eqref{eqwR} is in fact bounded in $W^{k+1,p}_{\rm{loc}}(B_4 \setminus \overline B_{1/4})$ uniformly with respect to $R\geq R_0$. 
Hence the linear $L^p$-theory yields the boundedness of $\{w_R\}_{R\geq R_0}$ in 
$W^{k+3,p}_{\rm{loc}}(B_4 \setminus \overline B_{1/4})$ for each $p<+\infty$. Then, by the Sobolev embedding, 
$\{\nabla w_R\}_{R\geq R_0}$ is bounded in $C^{k+1}_{\rm{loc}} (B_4 \setminus \overline B_{1/4})$, {\it i.e}, $(\mathcal{P}_{k+1}^\prime)$ holds. 
\vskip5pt

Now we prove that $(\mathcal{P}_k^\prime)$-$(\mathcal{P}_{k}^{\prime \prime})$  implies $(\mathcal{P}_{k+1}^{\prime\prime})$. 
We fix $\sigma \in (1,2)$ and we apply  $(\mathcal{P}_k^\prime)$, $ (\mathcal{P}_{k}^{\prime \prime})$ and  $ (\mathcal{P}_{k+1}^\prime)$ in 
$\overline B_{2\sigma^\prime}\setminus B_{1/2\sigma^\prime}$ for a fixed $\sigma<\sigma^\prime<2$, {\it e.g.}, $\sigma^\prime:=1+\sigma/2$. 
Since $K:=\overline B_{2\sigma}\setminus B_{1/2\sigma}$ is compact we can find finitely many points $\{ P_1, \ldots, P_m\} \subset K$ such that 
$K \subset \cup_{i=1}^m B_{\sigma^\prime-\sigma}(P_i)$ with $B_{2(\sigma^\prime -\sigma)}(P_i) \subset \overline B_{2\sigma^\prime}  \setminus B_{1/2 \sigma^\prime}$ for each $i=1,\ldots, m$. Then it suffices to show that $(\mathcal{P}_{k+1}^{\prime \prime})$ holds in each ball $B_i:=B_{\sigma^\prime-\sigma}(P_i)$ assuming  that 
$(\mathcal{P}_k^\prime)$, $ (\mathcal{P}_{k}^{\prime \prime})$ and  $ (\mathcal{P}_{k+1}^\prime)$ hold in $B^\prime_i:= B_{2(\sigma^\prime-\sigma)}(P_i)$. 
For simplicity we shall drop the subscript $i$. 

Taking \req{X_R} into account, we rewrite \req{eqrhoR} as
\begin{equation}
\label{eqX_R}
R^{-2} \Delta X_R=-\rho_R |\nabla w_R|^2+\rho_R (1+\rho_R) X_R \, .
\end{equation}
Denoting by $D^k$ any $k$-th derivative, since $\{\rho_R\}_{R\geq R_0}$, $\{X_R\}_{R\geq R_0}$, 
$\{w_R\}_{R\geq R_0}$ and $\{\nabla w_R\}_{R\geq R_0}$ are bounded in $C^k(\overline{B^\prime})$ 
by inductive assumption, differentiating \req{eqX_R} $k$-times leads to
$$\| D^k X_R\|_{L^\infty (B^\prime)}+ R^{-2} \| \Delta  D^k X_R \|_{L^\infty(B^\prime)} \leq C \, , $$
for some $C>0$ independent of $R\geq R_0$. 
Now we combine the above estimate with \cite[Lemma A.1]{BBH2} in $B \subset B^{\prime\prime} \subset B^\prime$ where $B^{\prime \prime}:=B_{3 (\sigma^\prime-\sigma)/2} (P_i)$ 
to obtain
\begin{equation} 
\label{derk+1X_R}
R^{-1} \| D^{k+1} X_R \|_{L^{\infty}(B^{\prime\prime})} \leq C \, 
\end{equation}
for a constant $C>0$ independent of $R\geq R_0$.
Finally we rewrite \req{eqX_R} as
\begin{equation}
\label{neweqX_R}
-R^{-2}\Delta X_R +2 X_R= 3 R^{-2}X_R^2-R^{-4} X_R^3 +\rho_R |\nabla w_R|^2=: \mathcal{T}_R \, .  
\end{equation}
As we already proved that $D^{k+1} \rho_R$ is bounded in $B^{\prime \prime}$ independently of $R\geq R_0$ and that $(\mathcal{P}_k^{\prime\prime})$, $(\mathcal{P}_{k+1}^\prime)$ hold in $B^{\prime\prime}$, taking \req{derk+1X_R} into account we infer that $f_R:=D^{k+1} \mathcal{T}_R$ satisfies  
$ \|  f_R\|_{L^{\infty}(B^{\prime\prime})} \leq C$ 
for a constant $C>0$ independent of $R\geq R_0$.
Then differentiating $(k+1)$-times \req{neweqX_R} we derive that $g_R:=D^{k+1} X_R$ satisfies 
\begin{equation}
\label{eqg_R}
\begin{cases}
-R^{-2} \Delta g_R +2 g_R =f_R  &  \text{in $B^{\prime\prime}$} \, ,\\
\| g_R\|_{L^\infty (B^{\prime\prime})} \leq CR \, , \\
\| f_R \|_{L^\infty (B^{\prime\prime})} \leq C \, ,
\end{cases}
\end{equation}  
for some $C>0$ independent of $R\geq R_0$. Next we write $g_R=\varphi_R +\psi_R$ in $\overline{B^{\prime\prime}}$ where $\varphi_R$ and $\psi_R$ are the 
unique smooth solutions of
\begin{equation}
\label{eqphi_R}
\begin{cases}
-R^{-2} \Delta \varphi_R +2 \varphi_R =0  &  \text{in $B^{\prime\prime}$} \, ,\\
\varphi_R= g_R  &  \text{on $\partial B^{\prime\prime}$} \, ,
\end{cases}
\end{equation}
and 
\begin{equation}
\label{eqpsi_R}
\begin{cases}
-R^{-2} \Delta \psi_R +2 \psi_R =f_R  &  \text{in $B^{\prime\prime}$} \, ,\\
\psi_R= 0 & \text{on $\partial B^{\prime\prime}$} \, .
\end{cases}
\end{equation}
Applying \cite[Lemma 2]{BBH2} in $B \subset B^{\prime\prime}$ to \req{eqphi_R}, 
the comparison principle in $B^{\prime\prime}$ to \req{eqpsi_R}, and the estimates in \req{eqg_R} we finally conclude 
$$\| D^{k+1} X_R \|_{L^\infty (B)}= \| g_R \|_{L^\infty(B)} \leq \| \varphi_R \|_{L^\infty(B)}+\| \psi_R \|_{L^\infty(B^{\prime\prime})}\leq C \,,$$
for some $C>0$ independent of $R\geq R_0$, {\it i.e.}, $(\mathcal{P}_{k+1}^{\prime\prime})$ holds in $B$.
\end{proof}

\begin{remark}
\label{C2compactness}
As a consequence of Proposition \ref{BBH}, Remark \ref{modtoonerate} and 
Proposition \ref{descripblowdown}, if $u$ is an entire solution to \eqref{GL} satisfying \req{1derbound}, then 
 $\{{u_R}_{|\mathbb{S}^2} \}_{R>0} $ is a compact 
subset of $C^2(\mathbb{S}^2;\mathbb{R}^3)$ and the limit as $R_n\to+\infty$ of any convergent sequence $\{{u_{R_n}}_{|\mathbb{S}^2} \}$ 
is an harmonic map $\omega\in C^2(\mathbb{S}^2;\mathbb{S}^2)$ (more precisely 
$\omega:=\phi_{|\mathbb{S}^2}$ where $\phi$ is given by Proposition \ref{descripblowdown}). 
In addition, for $n$ large 
the topological degree of ${u_{R_n}}_{|\mathbb{S}^2}$ is well defined and ${\rm deg}\,\omega={\rm deg}\,{u_{R_n}}_{|\mathbb{S}^2}={\rm deg}_{\infty}u\,$. 
\end{remark}

In order to prove uniqueness of the asymptotic limit of a solution $u$ at infinity,  
we need to establish a decay estimate on the radial derivative of $u$. 
As it will be clear below, such estimate gives the existence of a limit for the scaled maps $u_R$ as 
$R \to +\infty$ in $L^2(\mathbb{S}^2;\mathbb{R}^3)$. The a priori estimates in Proposition \ref{BBH}, 
as they yield compactness even in stronger topologies, 
will imply the convergence to an $\mathbb{S}^2$-valued harmonic map in $C^k(\mathbb{S}^2;\mathbb{R}^3)$ for any $k\in\NN$.

\begin{proposition}
\label{radderdecay}
Let $u$ be an entire solution of \req{GL} satisfiying \req{simonbound}. Then there exist $R_0\geq e$ and $C>0$ such that for any $R \geq R_0$, 
\begin{equation}
\label{radderinequality}
\int_{\{|x|>R\}} \frac{1}{|x|} \left| \frac{\partial u}{\partial r}\right|^2 dx \leq C\, \frac{\log R}{R^2} \, .
\end{equation}
\end{proposition}

\begin{proof}
By \req{simonbound} we can find $R_0 \geq e$ such that $|u(x)|\geq 1/2$ whenever $|x|\geq R_0$. 
Then we perform the polar decomposition of $u$, {\it i.e.}, for $|x|\geq R_0$ we write $u(x)=\rho(x) w(x)$ 
where $\rho(x)=|u(x)| \geq 1/2$ and $w(x) \in \mathbb{S}^2$. Due  to \req{simonbound} and 
\req{highderbound}, it is enough to prove \req{radderinequality} for $w$ 
since $\rho(x) \leq 1$ and $|\nabla \rho(x)|= \mathcal{O}(|x|^{-3})$ as $|x|\to +\infty$.  
Taking \req{polarderbounds} into account, we have $\nabla w(x)= \mathcal{O}(|x|^{-1})$ and 
$\Delta w(x)=\mathcal{O}(|x|^{-2})$ as $|x|\to +\infty$ so that equation  \req{polarsystem} can be rewritten as
\begin{equation}
\label{polareqw}
\Delta w(x) +w(x)|\nabla w(x)|^2=G(x) \, ,
\end{equation}
where 
$$G(x)=(1-\rho^2(x)) \big( \Delta w(x)+w(x) |\nabla w(x)|^2  \big)+ \nabla w(x) \cdot \nabla (1-\rho^2(x)) =\mathcal{O}(|x|^{-4}) $$
as $|x|\to+ \infty$ thanks  to \req{highderbound}. Next we multiply \req{polareqw} by $\ds \frac{\partial w}{\partial r}=\frac{x}{|x|} \cdot \nabla w$, 
and since $w$ and $\ds \frac{\partial w}{\partial r}$ are orthogonal, we obtain 
\begin{equation}
\label{radderdecay0}
 0=\left(\Delta w-G(x) \right) \cdot \frac{\partial w}{\partial r}=\frac{1}{|x|} \left| \frac{\partial w}{\partial r}\right|^2 +{\rm div}\, \Psi(x)-H(x) \, ,
 \end{equation}
where
$$\Psi(x)=\nabla w(x) \cdot \frac{\partial w}{\partial r}- \frac12|\nabla w(x)|^2 \frac{x}{|x|} \quad \hbox{and} \quad H(x)=G(x) \cdot \frac{\partial w}{\partial r} = \mathcal{O}(|x|^{-5}) $$
as $|x|\to +\infty$ by \req{simonbound}, \req{polarderbounds} and \req{highderbound}.
Integrating by parts \req{radderdecay0} in an annulus $A_{R^\prime,R}:=B_{R^\prime \setminus \overline{B_R}}$, with $R_0 \leq R< R^\prime$ gives
\begin{multline}
\label{radderdecay1}
\int_{A_{R^\prime,R}} \frac{1}{|x|} \left| \frac{\partial w}{\partial r}\right|^2 dx- 
\frac12 \int_{\partial B_R}\left| \frac{\partial w}{\partial r}\right|^2 d\mathcal{H}^2=  \frac{1}{2}\int_{\mathbb{S}^2}|\nabla_T\, w_{R^\prime}|^2d\mathcal{H}^2 
-\frac{1}{2}\int_{\mathbb{S}^2}|\nabla_T \,w_{R}|^2d\mathcal{H}^2+\\
+ \frac12 \int_{\partial B_R^\prime}\left| \frac{\partial w}{\partial r}\right|^2 d\mathcal{H}^2 +\int_{A_{R^\prime,R}} H\,dx \, ,
\end{multline}
where $w_R$ and $w_{R^\prime}$ are defined as Proposition \ref{BBH} and $\nabla_T$ denotes the tangential gradient.

Since \req{simonbound} obviously implies \eqref{lingro}, the Monotonicity Formula \req{monotonicity} yields 
$$\int_{\{|x|>R_0\}} \frac{1}{|x|} \left| \frac{\partial u}{\partial r}\right|^2 dx<+\infty\,.$$
Hence we can find a sequence $R^\prime_n \to +\infty$ such that 
\begin{equation}
\label{radderdecay2a}
 \int_{\partial B_{R^\prime_n}}  \left| \frac{\partial u}{\partial r}\right|^2 d\mathcal{H}^2 \mathop{\longrightarrow}\limits_{n\to+\infty} 0 \,.
\end{equation}
In view of Remark \ref{C2compactness} we can pass to a subsequence, still denoted by $\{R^\prime_n\}$, such that  
\begin{equation}
\label{radderdecay2} 
\| {u_{R_n^\prime}}_{|\mathbb{S}^2}-\omega \|_{C^2(\mathbb{S}^2;\mathbb{R}^3)} \mathop{\longrightarrow}\limits_{n\to+\infty} 0 \, ,
 \end{equation}
for some smooth harmonic map $\omega:\mathbb{S}^2 \to \mathbb{S}^2$ satisfying ${\rm deg} \, \omega= {\rm deg}_\infty u$.
Taking \req{simonbound} again into account, one may easily check that
\begin{equation}
\label{radderdecay3}
\int_{|x|>R_0}  \frac{1}{|x|}\left| \frac{\partial w}{\partial r}\right|^2dx<+\infty\,,\;\;
\int_{\partial B_{R^\prime_n}}  \left| \frac{\partial w}{\partial r}\right|^2 d\mathcal{H}^2 \mathop{\longrightarrow}\limits_{n\to +\infty} 0 \, ,
\;\; \int_{\mathbb{S}^2} |\nabla_T\,w_{R^\prime_n}|^2d\mathcal{H}^2 \mathop{\longrightarrow}\limits_{n\to +\infty} \int_{\mathbb{S}^2} |\nabla_T\,\omega|^2d\mathcal{H}^2  \, .
\end{equation}
 Choose $R^\prime=R^\prime_n$ in \req{radderdecay1}.  Taking \req{radderdecay3} into account and the integrability of $H$ 
 at infinity, we can pass to the limit $R^\prime_n\to+\infty$ to obtain 
 \begin{equation}
 \label{radderdecay4}
 \int_{\{|x|>R\}} \frac{1}{|x|} \left| \frac{\partial w}{\partial r}\right|^2 dx- \frac12 \int_{\partial B_R}\left| \frac{\partial w}{\partial r}\right|^2 d\mathcal{H}^2
 =   \frac{1}{2}\int_{\mathbb{S}^2} |\nabla_T\,\omega|^2d\mathcal{H}^2 -  \frac{1}{2}\int_{\mathbb{S}^2} |\nabla_T\,w_R|^2d\mathcal{H}^2+\int_{\{|x|>R\}} H\,dx \, ,
 \end{equation}
for each $R \geq R_0$. Then observe that ${\rm deg} \, {w_R}_{|\mathbb{S}^2}={\rm deg} \, \omega$ for each $R\geq R_0$ by Remark \ref{C2compactness}. On the other hand,  
$\omega:\mathbb{S}^2 \to \mathbb{S}^2$ is an harmonic map so that $\omega$ is energy minimizing in its own homotopy class. Therefore, 
\begin{equation}\label{enminhom}
\int_{\mathbb{S}^2} |\nabla_T\,\omega|^2d\mathcal{H}^2\leq\int_{\mathbb{S}^2} |\nabla_T\,w_R|^2d\mathcal{H}^2\,.
\end{equation}
Multiplying \req{radderdecay4} by $2R$ and using \eqref{enminhom}, we derive 
$$\frac{d}{dR} \left( R^2 \int_{\{|x|>R\}} \frac{1}{|x|} \left| \frac{\partial w}{\partial r}\right|^2dx \right)\leq 
2R \int_{\{|x|>R\}} H\, dx \, , $$
for every $R>R_0 $. Integrating the above inequality between $R_0$ and $R>R_0$,  using $H(x)=\mathcal{O}(|x|^{-5})$ and \eqref{radderdecay3},  
we finally obtain 
$$R^2 \int_{\{|x|>R\}} \frac{1}{|x|} \left| \frac{\partial w}{\partial r}\right|^2dx \leq
R_0^2 \intt{\{|x|>R_0\}} \frac{1}{|x|} \left| \frac{\partial w}{\partial r}\right|^2dx +C\int_{R_0}^R \frac{1}{r}dr \leq C( \log R +1)\, , $$
and the proof is complete.
\end{proof}

Now we are in a position to prove the  asymptotic symmetry of entire solutions of \req{GL}. 
\vskip5pt

\noindent {\rm \bf Proof of Theorem \ref{asymmetry}.}
Since $u$ satisfies \eqref{lingro} and $\{ u_R \}_{R>0}$ is relatively compact  in $H^1_{\rm{loc}}(\R^3;\R^3)$, 
we can apply Proposition \ref{Firstderbounds} and Proposition \ref{BBH} to obtain \req{simonbound}. Next 
we fix $R_0 $ as in Proposition \ref{radderdecay} and  we estimate  for $R_0\leq \tau_1 \leq \tau_2 \leq 2\tau_1$,
$$|u_{\tau_1}(\sigma)-u_{\tau_2}(\sigma)|^2\leq (\tau_2-\tau_1)\int_{\tau_1}^{\tau_2}\bigg|\frac{\partial u}{\partial r}(r\sigma)\bigg|^2dr\leq \int_{\tau_1}^{\tau_2}\bigg|\frac{\partial u}
{\partial r}(r\sigma)\bigg|^2rdr\quad\text{for every $\sigma\in\mathbb{S}^2$}\,. $$
Integrating the previous inequality with respect to $\sigma$, we infer from \eqref{radderinequality}
 that 
\begin{equation}\label{dyadic}
\int_{\mathbb{S}^2}|u_{\tau_1}-u_{\tau_2}|^2d\mathcal{H}^2\leq \int_{\{\tau_1\leq |x|\leq \tau_2\}} \frac{1}{|x|}\bigg|\frac{\partial u}{\partial r}\bigg|^2dx
\leq C\,\frac{\log\tau_1}{\tau_1^2}\quad\text{for every $R_0\leq \tau_1\leq\tau_2\leq 2\tau_1$}\,,
\end{equation}
where the constant $C$ only depends on $R_0$. 

Next we consider $R_0\leq R <R^\prime$ arbitrary. Define $k\in\NN$ to be the largest integer satisfying $2^kR\leq R^\prime$, and  
set  $\tau_j:=2^jR$ for $j=0,\ldots,k$ and $\tau_{k+1}:=R^\prime$. Using \eqref{dyadic} together with the triangle inequality, we estimate 
$$\| u_R-u_{R^\prime}\|_{L^2(\mathbb{S}^2)}\leq \sum_{j=0}^k\| u_{\tau_j}-u_{\tau_{j+1}}\|_{L^2(\mathbb{S}^2)}\leq C\sum_{j=0}^k\frac{\sqrt{\log\tau_j}}{\tau_j}\leq 
\frac{C}{R}\sum_{j=0}^{\infty} \frac{\sqrt{j\log 2 +\log R}}{2^j}\leq C\, \frac{\sqrt{\log R}}{R}\,,$$
for a constant $C$ which only depends on $R_0$. Obviously this estimate yields the uniqueness of the limit 
$\ds\omega:=\lim_{R \to +\infty} {u_R}_{|\mathbb{S}^2}$ in the $L^2$-topology. In view of Remark \ref{C2compactness} the convergence also holds in the $C^2$-topology and 
$\omega:\mathbb{S}^2\to\mathbb{S}^2$ is a smooth harmonic map satisfying ${\rm deg}\,\omega={\rm deg}_\infty u$.  So claim {\it (i)} in the theorem is proved. 
Then from claim {\it (i)}, \req{simonbound} and Proposition \ref{descripblowdown} we deduce that $u_R\to u_\infty$ strongly in 
$H^1_{\rm loc}(\R^3;\R^3)$ as $R\to+\infty$ with $u_\infty(x)=\omega(x/|x|)$, and claim {\it (ii)}. 

Moreover claim {\it (ii)} in Proposition \ref{descripblowdown} yields 
$$\int_{S^2} x |\nabla_T \,\omega | d\mathcal{H}^2=0\,.$$ 
As a consequence, if ${\rm deg}_\infty u=\pm 1={\rm deg} \, \omega $ the balancing condition above gives 
$\omega(x) = Tx$ for some $T \in O(3)$ by \cite[Proof of Theorem 7.3]{BCL}.  \prbox


\section{Proof of Theorem \ref{SYMMETRY}}

\noindent {\bf Proof of {\it (i)} $\Rightarrow$ {\it (ii)}.}  
This is just Theorem \ref{quantization}. 
\prbox
\vskip5pt

\noindent{\bf Proof of {\it (ii)} $\Rightarrow$ {\it (iii)}.} 
First we claim that the scaled maps $\{u_R\}_{R>0}$ given by \eqref{defscmap} are compact in $H^1_{\rm{loc}}(\R^3;\R^3)$. Indeed, by {\it (ii)} we can apply Proposition \ref{descripblowdown} to infer that 
from any weakly convergent sequence $\{u_{R_n}\}$ as $R_n\to+\infty$ we have 
$$\int_{B_1}\frac12 |\nabla \phi|^2 dx+\nu(B_1)=4\pi\,,$$
where $\phi$ is the weak limit of $\{u_{R_n}\}$ and $\nu$ is the defect measure as in Proposition \ref{descripblowdown}.  If 
$\nu \neq 0$ the above equality together with the structure of $\nu$ yields $\phi\equiv {\rm const}$ and $l=k_1=1$ which contraddicts the balancing condition in Proposition \ref{descripblowdown}, claim {\it (ii)}.  
Hence $\nu \equiv 0$ and $\{u_{R_n}\}$ is strongly convergent in $H^1_{\rm{loc}}(\R^3;\R^3)$.  

Now we can apply Theorem \ref{asymmetry} to get \req{simonbound} which obviously implies 
$|u(x)|=1+\mathcal{O}(|x|^{-2})$ as $|x|\to +\infty$.  Moreover $u_R\to u_\infty$ strongly in $H^1_{\rm loc}(\R^3;\R^3)$ as $R\to+\infty$ where $u_\infty(x)=\omega(x/|x|)$ for some  
smooth harmonic map $\omega:\mathbb{S}^2\to \mathbb{S}^2$ satisfying ${\rm deg}\,\omega={\rm deg}_\infty u$.  Therefore, 
$$4\pi |{\rm deg} \, \omega|=\int_{B_1} \frac12|\nabla u_\infty|^2dx = \lim_{R\to+ \infty} E_R(u_R,B_1)=\lim_{R\to+ \infty} \frac{1}{R} E(u,B_R)=4\pi \, , $$
so that ${\rm deg} \, \omega={\rm deg}_\infty u=\pm 1$. 
\prbox
\vskip5pt

\noindent{\bf Proof of {\it (iii)} $\Rightarrow$ {\it (iv)}.}  From Remark \ref{modtoonerate} we deduce that $u$ satisfies \eqref{lingro} 
and the scaled maps $\{ u_R \}_{R>0}$ are compact in $H^1_{\rm{loc}}(\R^3;\R^3)$.  
As a consequence we can apply Theorem \ref{asymmetry} to obtain estimate \req{simonbound}. 
In addition, up to an orthogonal transformation we may assume ${\rm deg}_\infty u=1$ and 
$\|u_R-{\rm Id}\|_{C^2(\mathbb{S}^2;\mathbb{R}^3)}\to 0$ as $R \to +\infty$. By degree theory we have $u^{-1}(\{0\}) \neq \emptyset$ and up to a translation, we may also assume that $u(0)=0$.   

Now we are in the position to apply the division trick of \cite{M2} (see also \cite{R} for another application). Let $f \in C^2([0,\infty))$ given by Lemma \ref{radode} 
and define 
$$v(x):=\frac{u(x)}{f(|x|)}\,.$$ 
Clearly $v \in C^2(\mathbb{R}^3 \setminus \{ 0\};\mathbb{R}^3)$, and it is straightforward to check that as $|x|\to 0$, 
\begin{equation}
\label{asymptorigin}
v(x)=B\,\frac{x}{|x|}+o(1) \quad\text{and} \quad \nabla v (x)=\nabla \bigg(B\, \frac{x}{|x|}\bigg)+ o(|x|^{-1}) \, , \quad\text{where $B:=\frac{\nabla u(0)}{f^\prime(0)}$} \, .
\end{equation} 
On the other hand, using Lemma \ref{radode} and the behaviour of $u$ at infinity, one may check that as $|x|\to +\infty$, 

\begin{equation}
\label{asymptinfinity}
v(x)=\frac{x}{|x|}+o(1) \, \qquad \nabla v(x)=\nabla \bigg(\frac{x}{|x|}\bigg)+o(|x|^{-1}) \, .
\end{equation}
Since $u$ solves \req{GL} and $f$ solves \req{cauchypb}, simple computations lead to 
$$\Delta v + f^2v(1-|v|^2)=-2\frac{f^\prime}{f}\,\frac{x}{|x|} \cdot \nabla v -\frac{2}{|x|^2}\,v \, .$$
Multiplying this equation by $\ds \frac{\partial v}{\partial r}=\frac{x}{|x|}\cdot \nabla v$ yields
\begin{equation}
\label{phozaevforv}
0\leq \left| \frac{\partial v}{\partial r}\right|^2\left( \frac{1}{|x|}+2 \frac{f^\prime}{f}\right)+\left( \frac{(1-|v|^2)^2}{4}\right) \left( 2 f f^\prime +\frac{2}{|x|}\right)={\rm div} \, \Phi(x) \, ,
\end{equation}
where
$$\Phi(x):= \left(  \frac12 |\nabla v|^2 \frac{x}{|x|}\right)- \left( \nabla v \cdot \frac{\partial v}{\partial r}\right) +\left( \frac{x}{|x|}f^2 \frac{(1-|v|^2)^2}{4} \right)+ \left( \frac{x}{|x|^3} (1-|v|^2) \right) \, . $$
Now we claim that 
\begin{equation}\label{lastclaim}
\int_{B_R \setminus B_\delta} {\rm div} \, \Phi\, dx=\int_{\{|x|=R\}} \Phi(x)\cdot \frac{x}{|x|}d\mathcal{H}^2- \int_{\{|x|=\delta\}} \Phi(x)\cdot \frac{x}{|x|}d\mathcal{H}^2\to 0 
\end{equation}
 as $R\to +\infty$ and $\delta \to 0$. Assume that the claim is proved. Then from \req{phozaevforv} we infer that $|v|\equiv 1$ and $\ds\frac{\partial v}{\partial r} \equiv 0$. 
 As a consequence, in view of \req{asymptinfinity} we derive that $|u(x)|\equiv f(|x|)$ and $v(x)\equiv x/|x|$ which concludes the proof.
\vskip5pt 

In order to prove \eqref{lastclaim} we first observe that as $|x|\to +\infty$, 
$$ |\nabla v|^2=\frac{2}{|x|^2}+o(|x|^{-2}) \, , \qquad \frac{\partial v}{\partial r}=o(|x|^{-1}) \, , \qquad 1-|v|^2=\mathcal{O}(|x|^{-2}) \, , $$
thanks to \req{asymptinfinity} and {\it (iii)}. Therefore, 
\begin{equation}
\label{divPhiinfinity}
\int_{\{|x|=R\}} \Phi(x)\cdot \frac{x}{|x|} d\mathcal{H}^2=\int_{\{|x|=R\}} \left( \frac{1}{|x|^2} +o(|x|^{-2}) \right)d\mathcal{H}^2=4\pi +o(1) \, \quad \hbox{as} \quad R\to +\infty \, .
\end{equation}
Next, using \req{asymptorigin},  we estimate as $|x|\to 0$, 
$$|\nabla v|^2=\bigg|\nabla\bigg( B\frac{x}{|x|}\bigg)\bigg|^2+o(|x|^{-2}) \,, \qquad \frac{\partial v}{\partial r}=o(|x|^{-1}) \, , \qquad 1-|v|^2=\frac{|x|^2-|Bx|^2}{|x|^2}+o(1) \, .  $$
Consequently, 
\begin{multline}
\label{divPhiorigin}
\int_{\{|x|=\delta\}} \Phi(x) \cdot \frac{x}{|x|} d\mathcal{H}^2= \int_{\{|x|=\delta\}} \left( \frac{1}{2} \bigg|\nabla \bigg(B\frac{x}{|x|}\bigg)\bigg|^2+ \frac{|x|^2-|Bx|^2}{|x|^4}+o(|x|^{-2})\right) d\mathcal{H}^2= \\
=\int_{\{|x|=1\}} \left( \frac{1}{2} \bigg|\nabla \bigg(B\frac{x}{|x|}\bigg)\bigg|^2- \frac{|Bx|^2}{|x|^4} \right) d\mathcal{H}^2 +4\pi+o(1) \quad \hbox{as} \quad \delta \to 0 \, . 
\end{multline}
Since   a direct computation gives 
$$\int_{\{|x|=1\}} \left(  \frac{1}{2} \bigg|\nabla \bigg(A\frac{x}{|x|}\bigg)\bigg|^2- \frac{|Ax|^2}{|x|^4} \right) d\mathcal{H}^2=0$$ 
for any constant matrix $A\in\mathbb{R}^{3\times 3}$, claim \eqref{lastclaim} follows combining \req{divPhiinfinity} and \req{divPhiorigin}. 
\prbox
\vskip5pt

\noindent{\bf Proof of {\it (iv)} $\Rightarrow$ {\it (i)}.}  Let $u$ be a nonconstant local minimizer as given by Theorem \ref{existence}. Since $R^{-1}E(u,B_R) \to 4\pi$ as $R \to+\infty$ and $u(0)=0$, and 
as we already proved {\it (ii)} $\Rightarrow$ {\it (iii)} $\Rightarrow {\it (iv)}$, we conclude that up to a rotation $u(x)=U(x)$ as given by \req{GLsolutions}. 
Hence $U$ is a nonconstant local minimizer of the energy, which is still the case when composing with translations and orthogonal transformations.
\prbox


\section*{Acknowledgments}
The authors would like to thank Fabrice Bethuel, Alberto Farina and Giovanni Leoni for useful discussions. This work was initiated while A.P. was
visiting the Carnegie Mellon University. He would like to thank Irene Fonseca  for the kind invitation and the warm hospitality. 
V.M. was partially supported by  the Center for
Nonlinear Analysis (CNA) under the National Science Fundation Grant
No. 0405343.


\end{document}